\documentclass[10pt]{article}

\usepackage[utf8]{inputenc}
\usepackage[T1]{fontenc}

\usepackage{epsf}
\usepackage{amsmath}

\allowdisplaybreaks

\usepackage[showframe=false]{geometry}
\usepackage{changepage}

\usepackage{epsfig}
\usepackage{amssymb}

\usepackage{amsthm}
\usepackage{setspace}
\usepackage{cite}
\usepackage{mcite}

\usepackage{algorithmic}  
\usepackage{algorithm}

\usepackage{shadow}
\usepackage{fancybox}
\usepackage{fancyhdr}

\usepackage{color}
\usepackage[usenames,dvipsnames,svgnames,table]{xcolor}
\newcommand{\bl}[1]{\textcolor{blue}{#1}}

\definecolor{mypurple}{rgb}{.4,.0,.5}

\usepackage[hyphens]{url}

\usepackage[colorlinks=true,
            linkcolor=black,
            urlcolor=blue,
            citecolor=purple]{hyperref}

\usepackage{breakurl}

\def\y{{\bf y}}

\def\x{{\bf x}}

\def\x{{\mathbf x}}

\def\u{{\bf u}}

\def\x{{\bf x}}
\def\y{{\bf y}}

\def\q{{\bf q}}
\def\m{{\bf m}}

\def\b{{\bf b}}
\def\c{{\bf c}}

\def\h{{\bf h}}

\def\cH{{\mathcal H}}

\def\be{\begin{equation}}
\def\ee{\end{equation}}
\def\ba{\left[\begin{array}}
\def\ea{\end{array}\right]}

\def\u{{\bf u}}

\def\x{{\bf x}}
\def\y{{\bf y}}

\def\q{{\bf q}}

\def\b{{\bf b}}
\def\c{{\bf c}}

\def\p{{\bf p}}

\def\1{{\bf 1}}

\def\0{{\bf 0}}

\def\erf{\mbox{erf}}
\def\erfc{\mbox{erfc}}

\def\calX{{\cal X}}
\def\calY{{\cal Y}}

\def\mR{{\mathbb R}}
\def\mN{{\mathbb N}}
\def\mE{{\mathbb E}}
\def\mS{{\mathbb S}}
\def\mP{{\mathbb P}}

\def\lp{\left (}
\def\rp{\right )}

\def\y{{\bf y}}

\def\x{{\bf x}}

\def\x{{\mathbf x}}

\def\u{{\bf u}}

\def\x{{\bf x}}
\def\y{{\bf y}}

\def\q{{\bf q}}

\def\b{{\bf b}}
\def\c{{\bf c}}

\def\h{{\bf h}}

\def\cH{{\cal H}}

\def\be{\begin{equation}}
\def\ee{\end{equation}}
\def\ba{\left[\begin{array}}
\def\ea{\end{array}\right]}

\def\u{{\bf u}}

\def\x{{\bf x}}
\def\y{{\bf y}}

\def\q{{\bf q}}

\def\b{{\bf b}}
\def\c{{\bf c}}

\def\p{{\bf p}}

\def\({\left (}
\def\){\right )}

\def\1{{\bf 1}}
\def\m{{\bf m}}
\def\q{{\bf q}}

\def\0{{\bf 0}}

\def\cX{{\mathcal X}}
\def\cY{{\mathcal Y}}

\usepackage{xcolor}
\usepackage{color}

\definecolor{darkgreen}{rgb}{0, 0.4,0}

\definecolor{purplebrown}{rgb}{0.5,0.1,0.6}

\definecolor{ultclupcol}{rgb}{0.1,0.5,0.5}

\definecolor{mytrycolor}{rgb}{0.5,0.7,0.2}

\definecolor{ultclupcola}{rgb}{.5,0,.5}

\definecolor{shadebrown}{rgb}{0.1,0.1,0.9}
\definecolor{lightblue}{rgb}{0.2,0,1}

\usepackage{fancybox}
\usepackage{graphicx}
\usepackage{epstopdf}
\usepackage{epsfig}
\usepackage{wrapfig}
\usepackage{subfigure}

\usepackage{xcolor}
\usepackage{tcolorbox}
\tcbuselibrary{skins}

\newtcbox{\xmybox}{on line,
arc=7pt,
before upper={\rule[-3pt]{0pt}{10pt}},boxrule=0pt,
boxsep=0pt,left=6pt,right=6pt,top=0pt,bottom=0pt,enhanced, coltext=blue, colback=white!10!yellow}

\newtcbox{\xmyboxa}{on line,
arc=7pt,
before upper={\rule[-3pt]{0pt}{10pt}},boxrule=0pt,
boxsep=0pt,left=6pt,right=6pt,top=0pt,bottom=0pt,enhanced, colback=white!10!yellow}

\newtcbox{\xmyboxb}{on line,
arc=7pt,
before upper={\rule[-3pt]{0pt}{10pt}},boxrule=1pt,colframe=darkgreen!100!blue,
boxsep=0pt,left=6pt,right=6pt,top=0pt,bottom=0pt,enhanced, colback=white!10!yellow}

\newtcbox{\xmyboxc}{on line,
arc=7pt,
before upper={\rule[-3pt]{0pt}{10pt}},boxrule=.7pt,colframe=blue!100!blue,
boxsep=0pt,left=6pt,right=6pt,top=0pt,bottom=0pt,enhanced, coltext=blue, colback=white!10!yellow}

\newtcbox{\xmytboxa}{on line,
arc=7pt,
before upper={\rule[-3pt]{0pt}{10pt}},boxrule=.0pt,colframe=pink!50!yellow,
boxsep=0pt,left=6pt,right=6pt,top=0pt,bottom=0pt,enhanced, coltext=white, colback=blue!40!red}

\newtcbox{\xmytboxb}{on line,
arc=7pt,
before upper={\rule[-3pt]{0pt}{10pt}},boxrule=.0pt,colframe=pink!50!yellow,
boxsep=0pt,left=6pt,right=6pt,top=0pt,bottom=0pt,enhanced, coltext=white, colback=white!40!green}

\makeatletter
\newcommand\subsubsubsection{\@startsection{paragraph}{4}{\z@}{-2.5ex\@plus -1ex \@minus -.25ex}{1.25ex \@plus .25ex}{\normalfont\normalsize\bfseries}}
\newcommand\subsubsubsubsection{\@startsection{subparagraph}{5}{\z@}{-2.5ex\@plus -1ex \@minus -.25ex}{1.25ex \@plus .25ex}{\normalfont\normalsize\bfseries}}
\makeatother

\newtheorem{theorem}{Theorem}

\newtheorem{corollary}{Corollary}

\begin{document}

\begin{singlespace}

\title {Binary perceptrons capacity via fully lifted random duality theory 
}
\author{
\textsc{Mihailo Stojnic
\footnote{e-mail: {\tt flatoyer@gmail.com}} }}
\date{}
\maketitle

\centerline{{\bf Abstract}} \vspace*{0.1in}

We study the statistical capacity of the classical binary perceptrons with general thresholds $\kappa$. After recognizing the connection between the capacity and the bilinearly indexed (bli) random processes, we utilize a recent progress in studying such processes to characterize the capacity. In particular, we rely on \emph{fully lifted} random duality theory (fl RDT) established in \cite{Stojnicflrdt23} to create a general framework for studying the perceptrons' capacities. Successful underlying numerical evaluations are required for the framework (and ultimately the entire fl RDT machinery) to become fully practically operational. We present results obtained in that directions and uncover that the capacity characterizations are achieved on the second (first non-trivial) level of \emph{stationarized} full lifting. The obtained results \emph{exactly} match the replica symmetry breaking predictions obtained through statistical physics replica methods in \cite{KraMez89}. Most notably, for the famous zero-threshold scenario, $\kappa=0$, we uncover the well known $\alpha\approx0.8330786$ scaled capacity.

\vspace*{0.25in} \noindent {\bf Index Terms: Binary perceptrons; Fully lifted random duality theory}.

\end{singlespace}

\section{Introduction}
\label{sec:back}

After a prior, more than half of a century long, intensive studying, it is the last two decades in particular that brought an unprecedented interest to  neural networks (NN) and machine learning (ML). Various research directions within these fields have been of interest, including, development of efficient algorithmic techniques, recognition of potential applications/connections to different scientific/engineering areas, creation of analytical tools for performance characterization and many others. Along similar lines, our primary interest here is the analytical progress and, in particular, theoretical studying of \emph{perceptrons} -- the moat basic NN units and almost unavoidable tools in any ML classifying concepts.

Due to their, basically irreplaceable, fundamental nature, perceptrons have been studied in many different fields and the amount of obtained results is rather large. It remains though completely fascinating that, even after quite a long period of studying and a plethora of great results achieved, many of the main concepts remain, from the analytical point, somewhat mysterious. Before we get to discuss some of them, we find it useful to briefly focus on the most relevant ones that are mathematically well understood. 

As expected, we start with the most typical (and in some sense the simplest)  perceptrons' variants -- the so-called \emph{spherical} perceptrons. Perhaps surprisingly, even such simple NN units, are not easy to completely analytically characterize. For example, their classifying/storage \emph{capacity} (as one of their most relevant features) is often difficult to compute and even harder to algorithmically achieve. Such a difficulty was recognized from the very early NN days. Consequently, the very early analytical treatments focused on various special cases and started with seminal works \cite{Wendel62,Winder,Cover65} where the zero-threshold capacity was precisely determined through combinatorial considerations closely related to high dimensional geometry and polytopal  neighborliness (for relevant geometric followup extensions see, e.g., \cite{Stojnicl1BnBxasymldp,Stojnicl1BnBxfinn}).

A generic analytical treatment started a couple of decades later though. Moreover, since the generic (non-zero thresholds) perceptrons turned out to be much harder than the zero-threshold ones, the early considerations relied on statistical physics insights. Unquestionably of greatest importance in that direction is the Gardner's seminal work, \cite{Gar88}, which paved the way for many of the very best perceptrons' analytical results that to this day stand as absolute state of the art. Namely, \cite{Gar88} and a follow-up \cite{GarDer88}, recognized and reaffirmed the role of a \emph{statistical approach} and consequently successfully adapted the, already well established statistical physics, replica approach so that it can treat almost any feature of various perceptrons models.  These included accurate predictions/approximations for their storage capacities in several different scenarios: positive/negative thresholds, correlated/uncorrelated patterns, patterns stored incorrectly and many others. Almost all of the predictions obtained in \cite{Gar88,GarDer88} were later on (in identical or similar statistical contexts) established rigorously as mathematical facts (see, e.g., \cite{SchTir02,SchTir03,Tal05,Talbook11a,Talbook11b,StojnicGardGen13,StojnicGardSphNeg13,StojnicGardSphErr13}). These typically related to the above mentioned spherical perceptrons. In particular, in \cite{SchTir02,SchTir03}, the authors confirmed the predictions made in \cite{Gar88} related to the storage capacity of the \emph{positive} spherical perceptrons. Moreover, the prediction related to the volume of the bond strengths that satisfies the perceptron dynamics was also confirmed as fully correct. Later on, in \cite{Tal05,Talbook11a,Talbook11b} Talagrand reconfirmed these predictions through a related but somewhat different approach. In \cite{StojnicGardGen13}, a completely different, random duality theory (RDT) based,  framework was introduced and, again, used to confirm almost all of the predictions from \cite{Gar88}, including many previously not considered in \cite{SchTir02,SchTir03,Tal05,Talbook11a,Talbook11b}. An underlying \emph{convexity} provided a substantial help in all of these, mathematically rigorous, treatments.

\subsection{Hard cases --- deviations from strong deterministic duality and convexity}
\label{sec:devsddconv}

As one starts facing scenarios where the strong deterministic duality does not hold, the above mathematically rigorous confirmations may become less effective. However, their power is still significant and, in almost all known instances, hard to beat. A prototypical such a scenario happens when the above positive spherical perceptrons transform into the \emph{negative} counterparts. This happens as soon as the perceptrons' thresholds move from $\kappa\geq 0$ to $\kappa<0$. The underlying deterministic strong duality is lost and obtaining accurate capacity characterizations becomes notoriously hard. Still, quite a lot of progress has been made when facing scenarios like this. In particular, \cite{StojnicGardSphNeg13}  ensured that the RDT power remains useful and proved the Talagrand's conjecture from \cite{Tal05,Talbook11a,Talbook11b} that the \cite{Gar88}'s results  related to the \emph{negative} spherical perceptrons are also rigorous upper bounds. Utilizing a \emph{lifted} RDT approach,  \cite{StojnicGardSphNeg13} then continued even further and established that, in certain range of problem parameters, these rigorous bounds can even be lowered (for more on the ensuing, algorithmic and replica based negative spherical perceptrons results, see excellent references \cite{FPSUZ17,FraHwaUrb19,FraPar16,FraSclUrb19,FraSclUrb20,AlaSel20}). Along the same lines, \cite{StojnicGardSphErr13} attacked spherical perceptrons when functioning as erroneous storage memories and again showed that the replica symmetry based predictions of \cite{Gar88} are rigorous upper bounds  which in certain range of system parameters can be lowered. Relying on the lifted RDT, both \cite{StojnicGardSphNeg13,StojnicGardSphErr13} effectively showed that, in the considered problems, the replica symmetry (assumed in \cite{Gar88,GarDer88}) must be broken.

Many other, so to say, ``\emph{hard}'' types of perceptrons have been intensively studied over the last several decades as well. Among the most challenging ones turned out to be the \emph{discrete} perceptrons. An introductory replica based treatment of a subclass of such perceptrons, the so-called binary $\pm 1$ perceptrons, was already given in foundational papers  \cite{Gar88,GarDer88} where it was promptly recognized  that handling  such perceptrons may be substantially harder than handling  the above discussed spherical ones. More specifically, the initial $\frac{4}{\pi}$ capacity prediction  in the simplest, zero-threshold ($\kappa=0$), case indicated that, for $\pm 1$ perceptrons, the framework of \cite{Gar88} may not be able to match even the simple combinatorics based alternative capacity characterizations that gave $1$ as an upper bound. Many other problems remained open as well. For example, while it was obvious that the capacity prediction of $\frac{4}{\pi}$ for  $\kappa=0$ is an upper bound, it was not clear if such a prediction can be safely made for all $\kappa$. Later on, \cite{StojnicDiscPercp13} considered general $\kappa$ thresholds and proved all key binary perceptrons replica predictions of \cite{Gar88,GarDer88}  as mathematically rigorous bounds.

As a result it was hinted that a more advanced version of the framework from \cite{Gar88} may be needed. Many great works followed attempting to resolve the problem. A couple of them relied on breaking the replica symmetry as the first next natural extension of \cite{Gar88,GarDer88}. The first study in this direction appeared in \cite{KraMez89} where, the now famous, zero-threshold ($\kappa=0$) capacity estimate $\approx 0.833$ was given (\cite{KraMez89} obtained estimates for any $\kappa$; it is, however, that $\approx 0.833$ is usually the single number most often associated with the binary perceptrons capacity). Similar arguments were then repeated in \cite{GutSte90} for $\pm 1$ perceptron and extended to $0/1$ and a few other discrete perceptrons  (several of the replica predictions of \cite{GutSte90} were, later on, shown  in \cite{StojnicDiscPercp13} as rigorous bounds or even fully accurate). Besides general $\kappa$ thresholds results of \cite{StojnicDiscPercp13}, the rigorous results also include, the zero-thresholds lower and upper bounds, $0.005$ and $0.9937$, of \cite{KimRoc98} (for alternative small lower bounds see also, e.g., \cite{BoltNakSunXu22,Tal99a}). Finally, the ``\emph{small probability}'' lower bound of \cite{DingSun19} (which was, later on, turned into a ``\emph{high probability}'' lower bound in \cite{NakSun23}) matches the $\approx 0.833$ prediction of \cite{KraMez89} (while the results of \cite{DingSun19,NakSun23,BoltNakSunXu22,Tal99a} were obtained for zero-thresholds $\kappa=0$, many of the underlying considerations hold for a general $\kappa$).

Many other great works are closely related as well and study various different relevant aspects. We mention some of them. In \cite{CXu21}, the existence of a capacity as a so-called sharp threshold was proven in the  Bernoulli models (this was extended to generic sub-gaussian ones in \cite{NakSun23}). Also, various different classes of binary perceptrons are of interest as well. Particularly related are the so-called \emph{symmetric binary} perceptrons. These problems are a bit easier and their capacity has been fully characterized (see, e.g., \cite{AbbLiSly21b,PerkXu21} as well as \cite{AbbLiSly21a,AlwLiuSaw21,AubPerZde19,GamKizPerXu22}).

We here first recognize the connection between the statistical perceptron problems and, in general, the so-called \emph{random feasibility problems} (rfps) on the one side and the \emph{random duality theory} (RDT) \cite{StojnicCSetam09,StojnicICASSP10var,StojnicRegRndDlt10,StojnicGardGen13,StojnicICASSP09} concepts on the other side. We then connect studying rfps to studying \emph{bilinearly indexed} (bli) random processes. Utilizing a recent progress made in studying bli processes in \cite{Stojnicsflgscompyx23,Stojnicnflgscompyx23}, in \cite{Stojnicflrdt23} a \emph{fully lifted} random duality theory (fl RDT) was established. Relying on the fl RDT and its a particular \emph{stationarized} fl RDT variant (called sfl RDT), we here obtain capacity characterizations. They become practically operational after underlying numerical evaluations are conducted. As a results of such evaluations, we find that on the second level of full lifting (2-sfl RDT), the obtained characterizations fully match the predictions of \cite{KraMez89} for any $\kappa$. In particular, for $\kappa=0$, we uncover the famous $\approx 0.8330786$ capacity result as well.

\section{Binary perceptrons as random feasibility problems (rfps)}
 \label{sec:bprfps}

We start with the following standard mathematical formulation of the \emph{feasibility} problems with linear inequalities
\begin{eqnarray}
\hspace{-1.5in}\mbox{Feasibility problem $\mathbf{\mathcal F}(G,\b,\cX,\alpha)$:} \hspace{1in}\mbox{find} & & \x\nonumber \\
\mbox{subject to}
& & G\x\geq \b \nonumber \\
& & \x\in\cX, \label{eq:ex1}
\end{eqnarray}
where $G\in\mR^{n\times n}$, $\b\in\mR^{m\times 1}$, $\cX\in\mR^n$, and $\alpha=\frac{m}{n}$. As recognized in \cite{StojnicGardGen13,StojnicGardSphErr13,StojnicGardSphNeg13,StojnicDiscPercp13}, the above formulation is precisely the same one that describes the simplest perceptrons. The perceptron's type is determined based on matrix $G$, vector $\b$, and set $\cX$. In particular, for
$\cX=\{-\frac{1}{\sqrt{n}},\frac{1}{\sqrt{n}} \}^b$ (i.e., for $\cX$ being the corners of the $n$-dimensional unit norm hypercube), we obtain the so-called binary $\pm 1$ perceptron with variable thresholds. For $\b=\kappa\1$, where $\kappa\in\mR$ and $\1$ (here and throughout the paper)  is a vector of all ones of appropriate dimensions, one further obtains the binary $\pm 1$ perceptron with \emph{fixed}, $\kappa$, threshold. Moreover, for a generic $G$, we have a deterministic perceptron whereas for a random $G$ a statistical one. Throughout the paper, we consider the so called Gaussian perceptrons where components of $G$ are i.i.d. standard normal random variables. This makes the presentation much neater and easier to follow. We however, mention that all of our results can easily be adapted to fit, in place of the standard normal, pretty much any other distribution that can be pushed through the Lindenberg variant of the central limit theorem.

One  can rewrite the feasibility problem from (\ref{eq:ex1}) as the following optimization problem
\begin{eqnarray}
\min_{\x} & & f(\x) \nonumber \\
\mbox{subject to}
& & G\x\geq \b \nonumber \\
& & \x\in\cX, \label{eq:ex1a1}
\end{eqnarray}
where an artificial function $f(\x):\mR^n \rightarrow\mR$ is introduced. Clearly, the above problem will be solvable only if it is actually feasible. Let us for a moment pretend that this is indeed the case. Then (\ref{eq:ex1a1}) can be rewritten as
\begin{eqnarray}
\xi_{feas}^{(0)}(f,\cX) = \min_{\x\in\cX} \max_{\y\in\cY_+}  \lp f(\x) -\y^T G\x +\y^T  \b  \rp,
 \label{eq:ex2}
\end{eqnarray}
where set $\cY_+$ is a collection of $\y$ such that $\y_{i}\geq 0,1\leq i\leq m$. Now, specializing back to $f(\x)=0$, one finds
\begin{eqnarray}
\xi_{feas}^{(0)}(0,\cX) = \min_{\x\in\cX} \max_{\y\in\cY_+}  \lp -\y^T G\x +\y^T  \b  \rp,
 \label{eq:ex2a1}
\end{eqnarray}
and observes that if there is an $\x$ such that $G\x\geq \b$, i.e., such that (\ref{eq:ex1}) is feasible, then the best that the inner maximization can do is make $\xi_{feas}^{(0)}(0,\cX) =0$. On the other hand, if there is no such an $\x$, then at least one of the inequalities in $G\x\geq \b$ is not satisfied and the inner maximization can trivially make $\xi_{feas}^{(0)}(0,\cX) =\infty$. It is not that difficult to see that from the feasibility point of view, $\xi_{feas}^{(0)}(0,\cX) =\infty$ and $\xi_{feas}^{(0)}(0,\cX) >0$ are equivalent (i.e., the problem is structurally insensitive with respect to $\y$ scaling). One can therefore restrict to $\|\y\|_2=1$ and ensure that $\xi_{feas}^{(0)}(0,\cX)$ remains bounded. From (\ref{eq:ex2a1}), it then easily follows that determining
\begin{eqnarray}
\xi_{feas}(0,\cX)
& =  &
\min_{\x\in\cX} \max_{\y\in\cY_+,\|\y\|_2=1}   \lp -\y^TG\x +\y^T\b \rp
                   =
\min_{\x\in\{-\frac{1}{\sqrt{n}},\frac{1}{\sqrt{n}}\}} \max_{\y\in\mS_+^m}  \lp -\y^TG\x + \kappa \y^T\1 \rp,
 \label{eq:ex3}
\end{eqnarray}
with $\mS_+^m$ being the positive orthant part of the $m$-dimensional unit sphere,
is of critical importance in characterizing the rfps from (\ref{eq:ex1}). Clearly, and as recognized in \cite{StojnicGardGen13,StojnicGardSphErr13,StojnicGardSphNeg13,StojnicDiscPercp13}, the positivity and non-positivity of the objective value in (\ref{eq:ex3}) (i.e., of $\xi_{feas}(f,\cX)$) ensure the infeasibility and feasibility of (\ref{eq:ex1}), respectively. As random \emph{feasibility} problems (rfps) (for example, the one from (\ref{eq:ex1})) are directly connected to perceptrons, handling their random \emph{optimization} problems counterparts (for example, the one from (\ref{eq:ex3})) is then of critical importance in determining various perceptrons' features. Of our interest here is, of course, the  \emph{capacity} of the perceptrons when used as storage memories or classifiers. In a large dimensional statistical context, such a capacity is defined as follows
 \begin{eqnarray}
\alpha & = &    \lim_{n\rightarrow \infty} \frac{m}{n}  \nonumber \\
\alpha_c(\kappa) & \triangleq & \max \{\alpha |\hspace{.08in}  \lim_{n\rightarrow\infty}\mP_G\lp\xi_{perc}(0,\cX)\triangleq \xi_{feas}(0,\cX)>0\rp\longrightarrow 1\} \nonumber \\
& = & \max \{\alpha |\hspace{.08in}  \lim_{n\rightarrow\infty}\mP_G\lp{\mathcal F}(G,\b,\cX,\alpha) \hspace{.07in}\mbox{is feasible} \rp\longrightarrow 1\}.
  \label{eq:ex4}
\end{eqnarray}
For the completeness, we also mention that the corresponding deterministic capacity variant is defined exactly as above with $\mP_G$ removed. Throughout the paper, the subscript next to $\mP$ (and later on $\mE$) denotes the randomness with respect to which the statistical evaluation is taken. On occasion, when this is clear from the contexts, the subscripts are left out.

\subsection{Random feasibility problems as free energy instances}
\label{secrfpsfe}

We start by introducing the so-called Hamiltonian
\begin{equation}
\cH_{sq}(G)= \y^TG\x,\label{eq:ham1}
\end{equation}
and the corresponding partition function
\begin{equation}
Z_{sq}(\beta,G)=\sum_{\x\in\cX} \lp \sum_{\y\in\cY}e^{\beta\cH_{sq}(G)}\rp^{-1},  \label{eq:partfun}
\end{equation}
where, for the overall generality, we take $\cX$ and $\cY$ as general sets (later on, we specialize to particular sets of our interest,  $\cX=\{-\frac{1}{\sqrt{n}},\frac{1}{\sqrt{n}}\}^n$ and $\cY=\mS_+^m$). The corresponding variant of the thermodynamic limit (average ``\emph{reciprocal}'') free energy is then
\begin{eqnarray}
f_{sq}(\beta) & = & \lim_{n\rightarrow\infty}\frac{\mE_G\log{(Z_{sq}(\beta,G)})}{\beta \sqrt{n}}
=\lim_{n\rightarrow\infty} \frac{\mE_G\log\lp \sum_{\x\in\cX} \lp \sum_{\y\in\cY}e^{\beta\cH_{sq}(G)}\rp^{-1}\rp}{\beta \sqrt{n}} \nonumber \\
& = &\lim_{n\rightarrow\infty} \frac{\mE_G\log\lp \sum_{\x\in\cX} \lp \sum_{\y\in\cY}e^{\beta\y^TG\x)}\rp^{-1}\rp}{\beta \sqrt{n}},\label{eq:logpartfunsqrt}
\end{eqnarray}
and its corresponding ground state special case
\begin{eqnarray}
f_{sq}(\infty)   \triangleq    \lim_{\beta\rightarrow\infty}f_{sq}(\beta) & = &
\lim_{\beta,n\rightarrow\infty}\frac{\mE_G\log{(Z_{sq}(\beta,G)})}{\beta \sqrt{n}}
=
 \lim_{n\rightarrow\infty}\frac{\mE_G \max_{\x\in\cX}  -  \max_{\y\in\cY} \y^TG\x}{\sqrt{n}} \nonumber \\
& = & - \lim_{n\rightarrow\infty}\frac{\mE_G \min_{\x\in\cX}  \max_{\y\in\cY} \y^TG\x}{\sqrt{n}}.
  \label{eq:limlogpartfunsqrta0}
\end{eqnarray}
Assuming that the components of $G$ are i.i.d. standard normals and given their sign symmetry we also have
\begin{eqnarray}
-f_{sq}(\infty)
& = &  \lim_{n\rightarrow\infty}\frac{\mE_G \min_{\x\in\cX}  \max_{\y\in\cY} \y^TG\x}{\sqrt{n}}  = \lim_{n\rightarrow\infty}\frac{\mE_G \min_{\x\in\cX}  \max_{\y\in\cY} -\y^TG\x}{\sqrt{n}}.
  \label{eq:limlogpartfunsqrt}
\end{eqnarray}
One now observes that $f_{sq}(\infty)$ is very tightly connected to $\xi_{feas}(0,\cX)$ given in (\ref{eq:ex3}). Understanding  $f_{sq}(\infty)$ is therefore critically important. However, studying $f_{sq}(\infty)$ directly is usually very hard. Instead, we rely on studying $f_{sq}(\beta)$, i.e., on studying the free energy defined for a general $\beta$ . The results of such a generic studying are eventually specialized to the so-called ground state behavior, $\beta\rightarrow\infty$. Consequently, the results  presented below easily account for any $\beta$. However, in the interest of easing the exposition,  some terms of no significance in the ground state regime are, on occasion, neglected.

\section{Fitting binary perceptrons into sfl RDT}
\label{sec:randlincons}

One first straightforwardly observes that the free energy from (\ref{eq:logpartfunsqrt}),
\begin{eqnarray}
f_{sq}(\beta) & = &\lim_{n\rightarrow\infty} \frac{\mE_G\log\lp \sum_{\x\in\cX} \lp \sum_{\y\in\cY}e^{\beta\y^TG\x)}\rp^{-1}\rp}{\beta \sqrt{n}},\label{eq:hmsfl1}
\end{eqnarray}
is a function of bli random process $\y^TG\x$. To establish a connection between $f_{sq}$ and the bli related results of \cite{Stojnicsflgscompyx23,Stojnicnflgscompyx23,Stojnicflrdt23}, we follow \cite{Stojnichopflrdt23}. To that end, we start with a few technical definitions. For $r\in\mN$, $k\in\{1,2,\dots,r+1\}$, real scalars $s$, $x$, and $y$  such that $s^2=1$, $x>0$, and $y>0$, sets $\cX\subseteq \mR^n$ and $\cY\subseteq \mR^m$, function $f_S(\cdot):\mR^n\rightarrow R$, vectors $\p=[\p_0,\p_1,\dots,\p_{r+1}]$, $\q=[\q_0,\q_1,\dots,\q_{r+1}]$, and $\c=[\c_0,\c_1,\dots,\c_{r+1}]$ such that
 \begin{eqnarray}\label{eq:hmsfl2}
1=\p_0\geq \p_1\geq \p_2\geq \dots \geq \p_r\geq \p_{r+1} & = & 0 \nonumber \\
1=\q_0\geq \q_1\geq \q_2\geq \dots \geq \q_r\geq \q_{r+1} & = &  0,
 \end{eqnarray}
$\c_0=1$, $\c_{r+1}=0$, and ${\mathcal U}_k\triangleq [u^{(4,k)},\u^{(2,k)},\h^{(k)}]$  such that the components of  $u^{(4,k)}\in\mR$, $\u^{(2,k)}\in\mR^m$, and $\h^{(k)}\in\mR^n$ are i.i.d. standard normals, we set
  \begin{eqnarray}\label{eq:fl4}
\psi_{S,\infty}(f_{S},\calX,\calY,\p,\q,\c,x,y,s)  =
 \mE_{G,{\mathcal U}_{r+1}} \frac{1}{n\c_r} \log
\lp \mE_{{\mathcal U}_{r}} \lp \dots \lp \mE_{{\mathcal U}_3}\lp\lp\mE_{{\mathcal U}_2} \lp \lp Z_{S,\infty}\rp^{\c_2}\rp\rp^{\frac{\c_3}{\c_2}}\rp\rp^{\frac{\c_4}{\c_3}} \dots \rp^{\frac{\c_{r}}{\c_{r-1}}}\rp, \nonumber \\
 \end{eqnarray}
where
\begin{eqnarray}\label{eq:fl5}
Z_{S,\infty} & \triangleq & e^{D_{0,S,\infty}} \nonumber \\
 D_{0,S,\infty} & \triangleq  & \max_{\x\in\cX,\|\x\|_2=x} s \max_{\y\in\cY,\|\y\|_2=y}
 \lp \sqrt{n} f_{S}
+\sqrt{n}  y    \lp\sum_{k=2}^{r+1}c_k\h^{(k)}\rp^T\x
+ \sqrt{n} x \y^T\lp\sum_{k=2}^{r+1}b_k\u^{(2,k)}\rp \rp \nonumber  \\
 b_k & \triangleq & b_k(\p,\q)=\sqrt{\p_{k-1}-\p_k} \nonumber \\
c_k & \triangleq & c_k(\p,\q)=\sqrt{\q_{k-1}-\q_k}.
 \end{eqnarray}
With all the above definitions in place, we can then recall on the following theorem -- certainly one of fundamental components of sfl RDT.
\begin{theorem} \cite{Stojnicflrdt23}
\label{thm:thmsflrdt1}  Consider large $n$ context with  $\alpha=\lim_{n\rightarrow\infty} \frac{m}{n}$, remaining constant as  $n$ grows. Let the elements of  $G\in\mR^{m\times n}$
 be i.i.d. standard normals and let $\cX\subseteq \mR^n$ and $\cY\subseteq \mR^m$ be two given sets. Assume the complete sfl RDT frame from \cite{Stojnicsflgscompyx23} and consider a given function $f(\y):R^m\rightarrow R$. Set
\begin{align}\label{eq:thmsflrdt2eq1}
   \psi_{rp} & \triangleq - \max_{\x\in\cX} s \max_{\y\in\cY} \lp f(\y)+\y^TG\x \rp
   \qquad  \mbox{(\bl{\textbf{random primal}})} \nonumber \\
   \psi_{rd}(\p,\q,\c,x,y,s) & \triangleq    \frac{x^2y^2}{2}    \sum_{k=2}^{r+1}\Bigg(\Bigg.
   \p_{k-1}\q_{k-1}
   -\p_{k}\q_{k}
  \Bigg.\Bigg)
\c_k
  - \psi_{S,\infty}(f(\y),\calX,\calY,\p,\q,\c,x,y,s) \hspace{.03in} \mbox{(\bl{\textbf{fl random dual}})}. \nonumber \\
 \end{align}
Let $\hat{\p_0}\rightarrow 1$, $\hat{\q_0}\rightarrow 1$, and $\hat{\c_0}\rightarrow 1$, $\hat{\p}_{r+1}=\hat{\q}_{r+1}=\hat{\c}_{r+1}=0$, and let the non-fixed parts of $\hat{\p}\triangleq \hat{\p}(x,y)$, $\hat{\q}\triangleq \hat{\q}(x,y)$, and  $\hat{\c}\triangleq \hat{\c}(x,y)$ be the solutions of the following system
\begin{eqnarray}\label{eq:thmsflrdt2eq2}
   \frac{d \psi_{rd}(\p,\q,\c,x,y,s)}{d\p} =  0,\quad
   \frac{d \psi_{rd}(\p,\q,\c,x,y,s)}{d\q} =  0,\quad
   \frac{d \psi_{rd}(\p,\q,\c,x,y,s)}{d\c} =  0.
 \end{eqnarray}
 Then,
\begin{eqnarray}\label{eq:thmsflrdt2eq3}
    \lim_{n\rightarrow\infty} \frac{\mE_G  \psi_{rp}}{\sqrt{n}}
  & = &
\min_{x>0} \max_{y>0} \lim_{n\rightarrow\infty} \psi_{rd}(\hat{\p}(x,y),\hat{\q}(x,y),\hat{\c}(x,y),x,y,s) \qquad \mbox{(\bl{\textbf{strong sfl random duality}})},\nonumber \\
 \end{eqnarray}
where $\psi_{S,\infty}(\cdot)$ is as in (\ref{eq:fl4})-(\ref{eq:fl5}).
 \end{theorem}
\begin{proof}
The $s=-1$ scenario follows directly from the corresponding one proven in \cite{Stojnicflrdt23} after a cosmetic change $f(\x)\rightarrow f(\y)$. On the other hand, the $s=1$ scenario, follows after trivial adjustments and a line-by-line repetition of the arguments of Section 3 of \cite{Stojnicflrdt23} with $s=-1$ replaced by $s=1$ and $f(\x)$ replaced by $f(\y)$.
 \end{proof}

 The above theorem is a generic result that relates to any given sets $\cX$ and $\cY$. The following corollary makes it operational for the case of binary $\pm 1$ perceptrons of interest here.
\begin{corollary}
\label{cor:cor1}  Assume the setup of Theorem \ref{thm:thmsflrdt1}. Let $\cX\subseteq \mR^n$ and $\cY\subseteq \mR^m$ be comprised of unit norm elements. Set
\begin{align}\label{eq:thmsflrdt2eq1a0}
   \psi_{rp} & \triangleq - \max_{\x\in\cX} s \max_{\y\in\cY} \lp \y^TG\x + \kappa \y^T\1 \rp
   \qquad  \mbox{(\bl{\textbf{random primal}})} \nonumber \\
   \psi_{rd}(\p,\q,\c,x,y,s) & \triangleq    \frac{1}{2}    \sum_{k=2}^{r+1}\Bigg(\Bigg.
   \p_{k-1}\q_{k-1}
   -\p_{k}\q_{k}
  \Bigg.\Bigg)
\c_k
  - \psi_{S,\infty}(\kappa\y^T\1,\calX,\calY,\p,\q,\c,1,1,s) \quad \mbox{(\bl{\textbf{fl random dual}})}. \nonumber \\
 \end{align}
Let the non-fixed parts of $\hat{\p}$, $\hat{\q}$, and  $\hat{\c}$ be the solutions of the following system
\begin{eqnarray}\label{eq:thmsflrdt2eq2a0}
   \frac{d \psi_{rd}(\p,\q,\c,1,1,s)}{d\p} =  0,\quad
   \frac{d \psi_{rd}(\p,\q,\c,1,1,s)}{d\q} =  0,\quad
   \frac{d \psi_{rd}(\p,\q,\c,1,1,s)}{d\c} =  0.
 \end{eqnarray}
 Then,
\begin{eqnarray}\label{eq:thmsflrdt2eq3a0}
    \lim_{n\rightarrow\infty} \frac{\mE_G  \psi_{rp}}{\sqrt{n}}
  & = &
 \lim_{n\rightarrow\infty} \psi_{rd}(\hat{\p},\hat{\q},\hat{\c},1,1,s) \qquad \mbox{(\bl{\textbf{strong sfl random duality}})},\nonumber \\
 \end{eqnarray}
where $\psi_{S,\infty}(\cdot)$ is as in (\ref{eq:fl4})-(\ref{eq:fl5}).
 \end{corollary}
\begin{proof}
Follows as a trivial consequence of Theorem \ref{thm:thmsflrdt1}, after taking $f(\y)=\kappa\y^T\1$ and recognizing that each element in $\cX$ and $\cY$ has unit norm.
 \end{proof}

As noted in \cite{Stojnicflrdt23,Stojnichopflrdt23}, various corresponding probabilistic variants of (\ref{eq:thmsflrdt2eq3}) and (\ref{eq:thmsflrdt2eq3a0}) follow through the above random primal problems' trivial concentrations. We skip stating these trivialities.

\section{Practical utilization}
\label{sec:prac}

Very elegant mathematical forms of Theorem \ref{thm:thmsflrdt1} and Corollary \ref{cor:cor1} are of practical use only if one can evaluate all the underlying quantities. As usual, two key obstacles might appear when trying to do so: (i) A lack of, a priori available, clarity as to what should be the correct value for $r$; and (ii) Set $\cY$ does not have a component-wise structure characterization which may render the decoupling over $\y$ as not necessarily overly straightforward. It turns out, however, that neither of these two obstacles poses a serious problem.

Specializing to $\cX=\{-\frac{1}{\sqrt{n}},\frac{1}{\sqrt{n}}\}^n$ and $\cY=\mS_+^m$ and relying on the results of Corollary \ref{cor:cor1}, we start by observing that the key object of practical interest is the so-called \emph{random dual}
\begin{align}\label{eq:prac1}
    \psi_{rd}(\p,\q,\c,1,1,s) & \triangleq    \frac{1}{2}    \sum_{k=2}^{r+1}\Bigg(\Bigg.
   \p_{k-1}\q_{k-1}
   -\p_{k}\q_{k}
  \Bigg.\Bigg)
\c_k
  - \psi_{S,\infty}(0,\calX,\calY,\p,\q,\c,1,1,s). \nonumber \\
  & =   \frac{1}{2}    \sum_{k=2}^{r+1}\Bigg(\Bigg.
   \p_{k-1}\q_{k-1}
   -\p_{k}\q_{k}
  \Bigg.\Bigg)
\c_k
  - \frac{1}{n}\varphi(D^{(bin)}(s)) - \frac{1}{n}\varphi(D^{(sph)}(s)), \nonumber \\
  \end{align}
where analogously to (\ref{eq:fl4})-(\ref{eq:fl5})
  \begin{eqnarray}\label{eq:prac2}
\varphi(D,\c) & = &
 \mE_{G,{\mathcal U}_{r+1}} \frac{1}{\c_r} \log
\lp \mE_{{\mathcal U}_{r}} \lp \dots \lp \mE_{{\mathcal U}_3}\lp\lp\mE_{{\mathcal U}_2} \lp
\lp    e^{D}   \rp^{\c_2}\rp\rp^{\frac{\c_3}{\c_2}}\rp\rp^{\frac{\c_4}{\c_3}} \dots \rp^{\frac{\c_{r}}{\c_{r-1}}}\rp, \nonumber \\
  \end{eqnarray}
and
\begin{eqnarray}\label{eq:prac3}
D^{(bin)}(s) & = & \max_{\x\in\cX} \lp   s\sqrt{n}      \lp\sum_{k=2}^{r+1}c_k\h^{(k)}\rp^T\x  \rp \nonumber \\
  D^{(sph)}(s) & \triangleq  &   s \max_{\y\in\cY}
\lp \sqrt{n} \kappa \y^T\1 + \sqrt{n}  \y^T\lp\sum_{k=2}^{r+1}b_k\u^{(2,k)}\rp \rp.
 \end{eqnarray}
A simple evaluation gives
\begin{eqnarray}\label{eq:prac4}
D^{(bin)}(s) & = & \max_{\x\in\cX}   \lp s\sqrt{n}      \lp\sum_{k=2}^{r+1}c_k\h^{(k)}\rp^T\x \rp =
       \sum_{i=1}^n \left |s\lp\sum_{k=2}^{r+1}c_k\h_i^{(k)}\rp \right |
=     \sum_{i=1}^n \left |\lp\sum_{k=2}^{r+1}c_k\h_i^{(k)}\rp \right |
=  \sum_{i=1}^n D^{(bin)}_i, \nonumber \\
 \end{eqnarray}
where
\begin{eqnarray}\label{eq:prac5}
D^{(bin)}_i(c_k)=\left |\lp\sum_{k=2}^{r+1}c_k\h_i^{(k)}\rp \right |.
\end{eqnarray}
Consequently,
  \begin{eqnarray}\label{eq:prac6}
\varphi(D^{(bin)}(s),\c) & = &
n \mE_{G,{\mathcal U}_{r+1}} \frac{1}{\c_r} \log
\lp \mE_{{\mathcal U}_{r}} \lp \dots \lp \mE_{{\mathcal U}_3}\lp\lp\mE_{{\mathcal U}_2} \lp
    e^{\c_2D_1^{(bin)}}  \rp\rp^{\frac{\c_3}{\c_2}}\rp\rp^{\frac{\c_4}{\c_3}} \dots \rp^{\frac{\c_{r}}{\c_{r-1}}}\rp
    = n\varphi(D_1^{(bin)}). \nonumber \\
   \end{eqnarray}
In a similar fashion, we also find
\begin{eqnarray}\label{eq:prac7}
   D^{(sph)}(s) & \triangleq  &   s \sqrt{n}  \max_{\y\in\cY}
\lp  \kappa \y^T\1 +   \y^T\lp\sum_{k=2}^{r+1}b_k\u^{(2,k)}\rp \rp
=  s  \sqrt{n}   \left \| \max \lp \kappa\1 +\sum_{k=2}^{r+1}b_k\u^{(2,k)},0 \rp  \right \|_2.
 \end{eqnarray}
Similarly to what was done in \cite{Stojnichopflrdt23},  we now utilize the  \emph{square root trick} introduced on numerous occasions in \cite{StojnicMoreSophHopBnds10,StojnicLiftStrSec13,StojnicGardSphErr13,StojnicGardSphNeg13}
\begin{align}\label{eq:prac8}
   D^{(sph)} (s)
& =   \sqrt{n}  s \left \| \max \lp \kappa\1 + \sqrt{n}\sum_{k=2}^{r+1}b_k\u^{(2,k)},0 \rp  \right \|_2
=  s\sqrt{n}  \min_{\gamma} \lp \frac{\left \| \max \lp \kappa\1 +  \sum_{k=2}^{r+1}b_k\u^{(2,k)},0 \rp  \right \|_2^2}{4\gamma}+\gamma \rp \nonumber \\
 & =   s\sqrt{n}  \min_{\gamma} \lp \frac{\sum_{i=1}^{m}  \max \lp \kappa +  \sum_{k=2}^{r+1}b_k\u_i^{(2,k)},0 \rp ^2}{4\gamma}+\gamma \rp.
 \end{align}
After introducing scaling $\gamma=\gamma_{sq}\sqrt{n}$, (\ref{eq:prac8}) can be rewritten as
\begin{eqnarray}\label{eq:prac9}
   D^{(sph)}(s)
  & =  & s\sqrt{n}  \min_{\gamma_{sq}} \lp \frac{\sum_{i=1}^{m} \max \lp \kappa + \sum_{k=2}^{r+1}b_k\u_i^{(2,k)},0 \rp^2}{4\gamma_{sq}\sqrt{n}}+\gamma_{sq}\sqrt{n} \rp   \nonumber \\
  & = & s \min_{\gamma_{sq}} \lp \frac{\sum_{i=1}^{m} \max \lp \kappa + \sum_{k=2}^{r+1}b_k\u_i^{(2,k)},0  \rp^2}{4\gamma_{sq}}+\gamma_{sq}n \rp \nonumber \\
  & =  &  s \min_{\gamma_{sq}} \lp \sum_{i=1}^{m} D_i^{(sph)}(b_k)+\gamma_{sq}n \rp, \nonumber \\
 \end{eqnarray}
with
\begin{eqnarray}\label{eq:prac10}
   D_i^{(sph)}(b_k)= \frac{\max \lp \kappa + \sum_{k=2}^{r+1}b_k\u_i^{(2,k)},0  \rp^2}{4\gamma_{sq}}.
 \end{eqnarray}

\subsection{Specialization to $s=-1$}
\label{sec:neg}

Taking $s=-1$ enables the connection between the ground state energy, $f_{sq}$ given in (\ref{eq:limlogpartfunsqrt}), and the random primal of the above machinery, $\psi_{rp}(\cdot)$, given in Corollary \ref{cor:cor1}. In particular, we find
 \begin{eqnarray}
-f_{sq}(\infty)
 & = &
- \lim_{n\rightarrow\infty}\frac{\mE_G \max_{\x\in\cX}  -  \max_{\y\in\cY} \y^TG\x}{\sqrt{n}}
 =
    \lim_{n\rightarrow\infty} \frac{\mE_G  \psi_{rp}}{\sqrt{n}}
   =
 \lim_{n\rightarrow\infty} \psi_{rd}(\hat{\p},\hat{\q},\hat{\c},1,1,-1),
  \label{eq:negprac11}
\end{eqnarray}
where the non-fixed parts of $\hat{\p}$, $\hat{\q}$, and  $\hat{\c}$ are the solutions of the following system
\begin{eqnarray}\label{eq:negprac12}
   \frac{d \psi_{rd}(\p,\q,\c,1,1,-1)}{d\p} =  0,\quad
   \frac{d \psi_{rd}(\p,\q,\c,1,1,-1)}{d\q} =  0,\quad
   \frac{d \psi_{rd}(\p,\q,\c,1,1,-1)}{d\c} =  0.
 \end{eqnarray}
Relying on (\ref{eq:prac1})-(\ref{eq:prac10}), we also have
 \begin{eqnarray}
 \lim_{n\rightarrow\infty} \psi_{rd}(\hat{\p},\hat{\q},\hat{\c},1,1,-1) =  \bar{\psi}_{rd}(\hat{\p},\hat{\q},\hat{\c},\hat{\gamma}_{sq},1,1,-1),
  \label{eq:negprac12a}
\end{eqnarray}
where
\begin{eqnarray}\label{eq:negprac13}
    \bar{\psi}_{rd}(\p,\q,\c,\gamma_{sq},1,1,-1)   & = &  \frac{1}{2}    \sum_{k=2}^{r+1}\Bigg(\Bigg.
   \p_{k-1}\q_{k-1}
   -\p_{k}\q_{k}
  \Bigg.\Bigg)
\c_k
\nonumber \\
& &  - \varphi(D_1^{(bin)}(c_k(\p,\q)),\c) +\gamma_{sq}- \alpha\varphi(-D_1^{(sph)}(b_k(\p,\q)),\c).
  \end{eqnarray}
A combination of  (\ref{eq:negprac11}), (\ref{eq:negprac12a}), and (\ref{eq:negprac13}) then gives
 \begin{eqnarray}
-f_{sq}(\infty)
& = &  -\lim_{n\rightarrow\infty}\frac{\mE_G \max_{\x\in\cX}  -  \max_{\y\in\cY} \y^TG\x}{\sqrt{n}} \nonumber \\
    &  = &
 \lim_{n\rightarrow\infty} \psi_{rd}(\hat{\p},\hat{\q},\hat{\c},1,1,-1)
 =   \bar{\psi}_{rd}(\hat{\p},\hat{\q},\hat{\c},\hat{\gamma}_{sq},1,1,-1) \nonumber \\
 & = &   \frac{1}{2}    \sum_{k=2}^{r+1}\Bigg(\Bigg.
   \hat{\p}_{k-1}\hat{\q}_{k-1}
   -\hat{\p}_{k}\hat{\q}_{k}
  \Bigg.\Bigg)
\hat{\c}_k
  - \varphi(D_1^{(bin)}(c_k(\hat{\p},\hat{\q})),\c) + \hat{\gamma}_{sq} - \alpha\varphi(-D_1^{(sph)}(b_k(\hat{\p},\hat{\q})),\c). \nonumber \\
  \label{eq:negprac18}
\end{eqnarray}
The above is summarized in the following theorem.

\begin{theorem}
  \label{thme:negthmprac1}
  Assume the complete sfl RDT setup of \cite{Stojnicsflgscompyx23}. Consider $\varphi(\cdot)$ and $\bar{\psi}(\cdot)$ from (\ref{eq:prac2}) and (\ref{eq:negprac13}) and large $n$ linear regime with $\alpha=\lim_{n\rightarrow\infty} \frac{m}{n}$. Let the ``fixed'' parts of $\hat{\p}$, $\hat{\q}$, and $\hat{\c}$ satisfy $\hat{\p}_1\rightarrow 1$, $\hat{\q}_1\rightarrow 1$, $\hat{\c}_1\rightarrow 1$, $\hat{\p}_{r+1}=\hat{\q}_{r+1}=\hat{\c}_{r+1}=0$, and let the ``non-fixed'' parts of $\hat{\p}_k$, $\hat{\q}_k$, and $\hat{\c}_k$ ($k\in\{2,3,\dots,r\}$) be the solutions of the following system of equations
  \begin{eqnarray}\label{eq:negthmprac1eq1}
   \frac{d \bar{\psi}_{rd}(\p,\q,\c,\gamma_{sq},1,1,-1)}{d\p} =  0 \nonumber \\
   \frac{d \bar{\psi}_{rd}(\p,\q,\c,\gamma_{sq},1,1,-1)}{d\q} =  0 \nonumber \\
   \frac{d \bar{\psi}_{rd}(\p,\q,\c,\gamma_{sq},1,1,-1)}{d\c} =  0 \nonumber \\
   \frac{d \bar{\psi}_{rd}(\p,\q,\c,\gamma_{sq},1,1,-1)}{d\gamma_{sq}} =  0,
 \end{eqnarray}
 and, consequently, let
\begin{eqnarray}\label{eq:prac17}
c_k(\hat{\p},\hat{\q})  & = & \sqrt{\hat{\q}_{k-1}-\hat{\q}_k} \nonumber \\
b_k(\hat{\p},\hat{\q})  & = & \sqrt{\hat{\p}_{k-1}-\hat{\p}_k}.
 \end{eqnarray}
 Then
 \begin{eqnarray}
-f_{sq}(\infty)
& = &     \frac{1}{2}    \sum_{k=2}^{r+1}\Bigg(\Bigg.
   \hat{\p}_{k-1}\hat{\q}_{k-1}
   -\hat{\p}_{k}\hat{\q}_{k}
  \Bigg.\Bigg)
\hat{\c}_k
  - \varphi(D_1^{(bin)}(c_k(\hat{\p},\hat{\q})),\hat{\c}) + \hat{\gamma}_{sq} - \alpha\varphi(-D_1^{(sph)}(b_k(\hat{\p},\hat{\q})),\hat{\c}). \nonumber \\
  \label{eq:negthmprac1eq2}
\end{eqnarray}
\end{theorem}
\begin{proof}
Follows from the previous discussion, Theorem \ref{thm:thmsflrdt1}, Corollary \ref{cor:cor1}, and the sfl RDT machinery presented in \cite{Stojnicnflgscompyx23,Stojnicsflgscompyx23,Stojnicflrdt23,Stojnichopflrdt23}.
\end{proof}

\subsubsection{Numerical evaluations}
\label{sec:nuemrical}

The results of Theorem \ref{thme:negthmprac1} are sufficient to conduct numerical evaluations. Several analytical results can be obtained and we state them below as well. To ensure a systematic view regarding progressing of the whole lifting mechanism, we start the numerical evaluations with $r=1$ and proceed inductively. To enable concrete numerical values, we, on occasion specialize the evaluation to the most famous, zero-threshold, $\kappa=0$, case.

\subsubsubsection{$r=1$ -- first level of lifting}
\label{sec:firstlev}

For $r=1$ one has that $\hat{\p}_1\rightarrow 1$ and $\hat{\q}_1\rightarrow 1$ which together with $\hat{\p}_{r+1}=\hat{\p}_{2}=\hat{\q}_{r+1}=\hat{\q}_{2}=0$, and $\hat{\c}_{2}\rightarrow 0$ gives
\begin{align}\label{eq:negprac19}
    \bar{\psi}_{rd}(\hat{\p},\hat{\q},\hat{\c},\gamma_{sq},1,1,-1)   & =   \frac{1}{2}
\c_2
  - \frac{1}{\c_2}\log\lp \mE_{{\mathcal U}_2} e^{\c_2|\sqrt{1-0}\h_1^{(2)} |}\rp +\gamma_{sq}
- \alpha\frac{1}{\c_2}\log\lp \mE_{{\mathcal U}_2} e^{-\c_2\frac{\max(\kappa+\sqrt{1-0}\u_1^{(2,2)},0)^2}{4\gamma_{sq}}}\rp \nonumber \\
& \rightarrow
  - \frac{1}{\c_2}\log\lp 1+ \mE_{{\mathcal U}_2} \c_2|\sqrt{1-0}\h_1^{(2)} |\rp +\gamma_{sq} \nonumber \\
& \qquad - \alpha\frac{1}{\c_2}\log\lp 1- \mE_{{\mathcal U}_2} \c_2\frac{\max(\kappa+\sqrt{1-0}\u_1^{(2,2)},0)^2}{4\gamma_{sq}}\rp \nonumber \\
& \rightarrow
   - \frac{1}{\c_2}\log\lp 1+ \c_2\sqrt{\frac{2}{\pi}}\rp +\gamma_{sq} \nonumber \\
& \qquad - \alpha\frac{1}{\c_2}\log\lp 1- \frac{\c_2}{4\gamma_{sq}} \mE_{{\mathcal U}_2} \max(\kappa+\sqrt{1-0}\u_1^{(2,2)},0)^2 \rp \nonumber \\
& \rightarrow
  - \sqrt{\frac{2}{\pi}}+\gamma_{sq}
+  \frac{\alpha}{4\gamma_{sq}}\mE_{{\mathcal U}_2} \max(\kappa+\sqrt{1-0}\u_1^{(2,2)},0)^2.
  \end{align}
One then easily finds $\hat{\gamma}_{sq}=\frac{\sqrt{\alpha}}{2}\sqrt{\mE_{{\mathcal U}_2} \max(\kappa+\sqrt{1-0}\u_1^{(2,2)},0)^2}$ and
\begin{align}\label{eq:negprac20}
 - f_{sq}^{(1)}(\infty)=\bar{\psi}_{rd}(\hat{\p},\hat{\q},\hat{\c},\hat{\gamma}_{sq},1,1,-1)   & =
  - \sqrt{\frac{2}{\pi}}+\sqrt{\alpha}\sqrt{\mE_{{\mathcal U}_2} \max(\kappa+\u_1^{(2,2)},0)^2}.
  \end{align}
The condition $f_{sq}^{(1)}(\infty)=0$ gives the critical $\alpha_c^{(1)}$ as a function of $\kappa$
\begin{equation}\label{eq:negprac20a0}
a_c^{(1)}(\kappa)
=  \frac{2}{\pi\mE_{{\mathcal U}_2} \max(\kappa+\u_1^{(2,2)},0)^2}
=  \frac{2}{\pi\lp \frac{\kappa e^{-\frac{\kappa^2}{2}}}{\sqrt{2\pi}} + \frac{(\kappa^2+1)\erfc\lp -\frac{\kappa}{\sqrt{2}} \rp}{2}  \rp},
  \end{equation}
which for $\kappa=0$ becomes
\begin{equation}\label{eq:negprac21}
\hspace{-2in}(\mbox{first level:}) \qquad \qquad  a_c^{(1)}(0) =
  \frac{2}{\pi\mE_{{\mathcal U}_2} \max(\u_1^{(2,2)},0)^2} =  \frac{2}{\pi\frac{1}{2}} = \frac{4}{\pi}
\rightarrow  \bl{\mathbf{1.2732}}.
  \end{equation}

\subsubsubsection{$r=2$ -- second level of lifting}
\label{sec:secondlev}

We split the second level of lifting into two subparts: (i) \emph{partial} second level of lifting; and (ii) \emph{full} second level of lifting.

\underline{1) \textbf{\emph{Partial second level of lifting:}}}  For this part, we have, $r=2$,  $\hat{\p}_1\rightarrow 1$ and $\hat{\q}_1\rightarrow 1$, $\hat{\p}_{2}=\hat{\q}_{2}=0$, and $\hat{\p}_{r+1}=\hat{\p}_{3}=\hat{\q}_{r+1}=\hat{\q}_{3}=0$ but in general  $\hat{\c}_{2}\neq 0$. As above, one again has
\begin{align}\label{eq:negprac22}
    \bar{\psi}_{rd}(\hat{\p},\hat{\q},\c,\gamma_{sq},1,1,-1)   & =   \frac{1}{2}
\c_2
  - \frac{1}{\c_2}\log\lp \mE_{{\mathcal U}_2} e^{\c_2|\sqrt{1-0}\h_1^{(2)} |}\rp + \gamma_{sq}
- \alpha\frac{1}{\c_2}\log\lp \mE_{{\mathcal U}_2} e^{-\c_2\frac{\max(\kappa+\sqrt{1-0}\u_1^{(2,2)},0)^2}{4\gamma_{sq}}}\rp \nonumber \\
& =   \frac{1}{2}
\c_2
  - \frac{1}{\c_2}\log \lp e^{\frac{\c_2^2}{2}}\erfc\lp -\frac{\c_2}{\sqrt{2}}\rp \rp + \gamma_{sq}
- \alpha\frac{1}{\c_2}\log\lp \mE_{{\mathcal U}_2} e^{-\c_2\frac{\max(\kappa+\sqrt{1-0}\u_1^{(2,2)},0)^2}{4\gamma_{sq}}}\rp \nonumber \\
& =
  - \frac{1}{\c_2}\log \lp  \erfc\lp -\frac{\c_2}{\sqrt{2}}\rp \rp + \gamma_{sq}
- \alpha\frac{1}{\c_2}\log\lp \mE_{{\mathcal U}_2} e^{-\c_2\frac{\max(\kappa+\sqrt{1-0}\u_1^{(2,2)},0)^2}{4\gamma_{sq}}}\rp. \nonumber \\
   \end{align}
Solving the integrals gives
\begin{eqnarray}\label{eq:negprac22a0}
\bar{h} & = & -\kappa    \nonumber \\
\bar{B} & = & \frac{\c_2}{4\gamma_{sq}} \nonumber \\
\bar{C} & = & \kappa \nonumber \\
f_{(zd)}& = & \frac{e^{-\frac{\bar{B}\bar{C}^2}{2\bar{B} + 1}}}{2\sqrt{2\bar{B} + 1}}\erfc\lp\frac{\bar{h}}{\sqrt{4\bar{B} + 2}}\rp  \nonumber \\
f_{(zu)}& = & \frac{1}{2}\erfc\lp-\frac{\bar{h}}{\sqrt{2}}\rp,
   \end{eqnarray}
 and
\begin{equation}\label{eq:negprac22a1}
   \mE_{{\mathcal U}_2} e^{-\c_2\frac{\max(\kappa+\sqrt{1-0}\u_1^{(2,2)},0)^2}{4\gamma_{sq}}}=f_{(zd)} + f_{(zu)}.
   \end{equation}
After differentiating  (optimizing) with respect to $\gamma_{sq}$ and $\c_2$,  we distinguish two scenarios depending on the concrete value for $\kappa$.

\noindent (\textbf{\emph{i}}) For $\kappa\geq \kappa_c$ (where $\kappa_c$ will be specified below), $\hat{\c}_2\rightarrow 0$ and $\hat{\gamma}_{sq}=\frac{\sqrt{\alpha}}{2}\sqrt{\mE_{{\mathcal U}_2} \max(\kappa+\sqrt{1-0}\u_1^{(2,2)},0)^2}$, i.e., one uncovers the first level of lifting with $a_c^{(2,p)}$ as in (\ref{eq:negprac21}), i.e., with
\begin{equation}\label{eq:negprac22a1a0}
a_c^{(2,p)} =
a_c^{(1)} =
  \frac{2}{\pi\mE_{{\mathcal U}_2} \max(\kappa+\u_1^{(2,2)},0)^2}
  = \frac{2}{\pi\lp \frac{\kappa e^{-\frac{\kappa^2}{2}}}{\sqrt{2\pi}} + \frac{(\kappa^2+1)\erfc\lp -\frac{\kappa}{\sqrt{2}} \rp}{2}  \rp}.
  \end{equation}

\noindent (\textbf{\emph{ii}}) For $\kappa\leq \kappa_c$, $\hat{\c}_2\rightarrow \infty$ and $\hat{\gamma}_{sq}\rightarrow 0$. (\ref{eq:negprac22a0}) then becomes
\begin{eqnarray}\label{eq:negprac22a2}
\bar{h} & = & -\kappa    \nonumber \\
\bar{B} & = & \frac{\c_2}{4\gamma_{sq}} \rightarrow \infty \nonumber \\
\bar{C} & = & \kappa \nonumber \\
f_{(zd)}^{(2,p)}& = & \frac{e^{-\frac{\bar{B}\bar{C}^2}{2\bar{B} + 1}}}{2\sqrt{2\bar{B} + 1}}\erfc\lp\frac{\bar{h}}{\sqrt{4\bar{B} + 2}}\rp
\rightarrow 0 \nonumber \\
f_{(zu)}^{(2,p)}& = & \frac{1}{2}\erfc\lp-\frac{\bar{h}}{\sqrt{2}}\rp  \rightarrow \frac{1}{2}\erfc\lp \frac{\kappa}{\sqrt{2}}\rp,
   \end{eqnarray}
 which then transforms (\ref{eq:negprac22a1}) into
\begin{equation}\label{eq:negprac22a3}
   \mE_{{\mathcal U}_2} e^{-\c_2\frac{\max(\kappa+\sqrt{1-0}\u_1^{(2,2)},0)^2}{4\gamma_{sq}}}=f_{(zd)}^{(2,p)} + f_{(zu)}^{(2,p)} =
   \frac{1}{2}\erfc\lp \frac{\kappa}{\sqrt{2}}\rp.
   \end{equation}
Combining (\ref{eq:negprac22}) and (\ref{eq:negprac22a3}), we find
\begin{eqnarray}\label{eq:negprac22a4}
-f_{sq}^{(2,p)}(\infty)  & = &  \bar{\psi}_{rd}(\hat{\p},\hat{\q},\c,\gamma_{sq},1,1,-1)    \nonumber \\
& = &
  - \frac{1}{\c_2}\log \lp  \erfc\lp -\frac{\c_2}{\sqrt{2}}\rp \rp + \gamma_{sq}
- \alpha\frac{1}{\c_2}\log\lp \mE_{{\mathcal U}_2} e^{-\c_2\frac{\max(\kappa+\sqrt{1-0}\u_1^{(2,2)},0)^2}{4\gamma_{sq}}}\rp. \nonumber \\
 &\rightarrow &
  - \frac{1}{\c_2}\log \lp 2 \rp
- \alpha\frac{1}{\c_2}\log\lp    \frac{1}{2}\erfc\lp \frac{\kappa}{\sqrt{2}}\rp \rp.
    \end{eqnarray}
As earlier, the condition $f_{sq}^{(2,p)}(\infty)=0$ gives the critical $\alpha_c^{(2,p)}$
\begin{equation}\label{eq:negprac22a5}
a_c^{(2,p)} =
 - \frac{\log(2)}{\log \lp \frac{1}{2}\erfc\lp \frac{\kappa}{\sqrt{2}}\rp \rp}.
  \end{equation}
From (\ref{eq:negprac22a1a0})  and  (\ref{eq:negprac22a5}), we find $\kappa_c$ as the one that satisfies
\begin{equation}\label{eq:negprac22a6}
    \frac{2}{\pi\mE_{{\mathcal U}_2} \max(\kappa_c+\u_1^{(2,2)},0)^2}=
    \frac{2}{\pi\lp \frac{\kappa_c e^{-\frac{\kappa_c^2}{2}}}{\sqrt{2\pi}} + \frac{(\kappa_c^2+1)\erfc\lp -\frac{\kappa_c}{\sqrt{2}} \rp}{2}  \rp}
    =- \frac{\log(2)}{\log \lp \frac{1}{2}\erfc\lp \frac{\kappa_c}{\sqrt{2}}\rp \rp}.
  \end{equation}
Numerical evaluation gives $\kappa_c=0.602957...$ and
\begin{equation}\label{eq:negprac22a7}
a_c^{(2,p)}(\kappa) =\begin{cases}
    \frac{2}{\pi\lp \frac{\kappa e^{-\frac{\kappa^2}{2}}}{\sqrt{2\pi}} + \frac{(\kappa^2+1)\erfc\lp -\frac{\kappa}{\sqrt{2}} \rp}{2}  \rp}, & \mbox{if } \kappa\geq \kappa_c=0.602957... \\
             - \frac{\log(2)}{\log \lp \frac{1}{2}\erfc\lp \frac{\kappa}{\sqrt{2}}\rp \rp}, & \mbox{otherwise}.
           \end{cases}
  \end{equation}
Specializing to $\kappa=0$, we find
\begin{equation}\label{eq:negprac23}
\hspace{-2in}(\mbox{\emph{partial} second level:}) \qquad \qquad  a_c^{(2,p)}(0) =
 - \frac{\log(2)}{\log \lp \frac{1}{2}  \rp} = \bl{\mathbf{1}}.
  \end{equation}

\underline{2) \textbf{\emph{Full second level of lifting:}}}  For this part, we utilize the same setup as above with the exception that now (in addition to $\hat{\c}_{2}\neq 0$) one, in general,  also has $\p_2\neq0$ and $\q_2\neq0$. Analogously to (\ref{eq:negprac22}), we now have
\begin{eqnarray}\label{eq:negprac24}
    \bar{\psi}_{rd}(\p,\q,\c,\gamma_{sq},1,1,-1)   & = &  \frac{1}{2}
(1-\p_2\q_2)\c_2
  - \frac{1}{\c_2}\mE_{{\mathcal U}_3}\log\lp \mE_{{\mathcal U}_2} e^{\c_2|\sqrt{1-\q_2}\h_1^{(2)} +\sqrt{\q_2}\h_1^{(3)} |}\rp \nonumber \\
& &   + \gamma_{sq}
 -\alpha\frac{1}{\c_2}\mE_{{\mathcal U}_3} \log\lp \mE_{{\mathcal U}_2} e^{-\c_2\frac{\max(\sqrt{1-\p_2}\u_1^{(2,2)}+\sqrt{\p_2}\u_1^{(2,3)},0)^2}{4\gamma_{sq}}}\rp.
    \end{eqnarray}
One then first finds
\begin{eqnarray}\label{eq:negprac24a0}
f_{(z)}^{(2)} & = & \mE_{{\mathcal U}_2} e^{\c_2|\sqrt{1-\q_2}\h_1^{(2)} +\sqrt{\q_2}\h_1^{(3)} |}
  \nonumber \\
 & = &  \frac{1}{2}
 e^{\frac{(1-\q_2)\c_2^2}{2}}
 \Bigg(\Bigg.
 e^{-\c_2\sqrt{\q_2}\h_1^{(3)}}
 \erfc\lp - \lp\c_2\sqrt{1-\q_2}-\frac{\sqrt{\q_2}\h_1^{(3)}}{\sqrt{1-\q_2}}\rp\frac{1}{\sqrt{2}}\rp \nonumber \\
& &  + e^{\c_2\sqrt{\q_2}\h_1^{(3)}}
   \erfc\lp - \lp\c_2\sqrt{1-\q_2}+\frac{\sqrt{\q_2}\h_1^{(3)}}{\sqrt{1-\q_2}}\rp\frac{1}{\sqrt{2}}\rp
   \Bigg.\Bigg),
     \end{eqnarray}
and
\begin{eqnarray}\label{eq:negprac24a1}
  \mE_{{\mathcal U}_3}\log\lp \mE_{{\mathcal U}_2} e^{\c_2|\sqrt{1-\q_2}\h_1^{(2)} +\sqrt{\q_2}\h_1^{(3)} |}\rp
=  \mE_{{\mathcal U}_3}\log\lp f_{(z)}^{(2)}\rp.
    \end{eqnarray}
Moreover, we also have
\begin{eqnarray}\label{eq:negprac24a2}
\bar{h} & = &  -\frac{\sqrt{\p_2}\u_1^{(2,3)}+\kappa}{\sqrt{1-\p_2}}    \nonumber \\
\bar{B} & = & \frac{\c_2}{4\gamma_{sq}} 
\nonumber \\
\bar{C} & = & \sqrt{\p_2}\u_1^{(2,3)}+\kappa \nonumber \\
f_{(zd)}^{(2,f)}& = & \frac{e^{-\frac{\bar{B}\bar{C}^2}{2(1-\p_2)\bar{B} + 1}}}{2\sqrt{2(1-\p_2)\bar{B} + 1}}
\erfc\lp\frac{\bar{h}}{\sqrt{4(1-\p_2)\bar{B} + 2}}\rp
\nonumber \\
f_{(zu)}^{(2,f)}& = & \frac{1}{2}\erfc\lp-\frac{\bar{h}}{\sqrt{2}}\rp,  
   \end{eqnarray}
and
\begin{eqnarray}\label{eq:negprac24a3}
   \mE_{{\mathcal U}_3} \log\lp \mE_{{\mathcal U}_2} e^{-\c_2\frac{\max(\sqrt{1-\p_2}\u_1^{(2,2)}+\sqrt{\p_2}\u_1^{(2,3)},0)^2}{4\gamma_{sq}}}\rp
=   \mE_{{\mathcal U}_3} \log\lp f_{(zd)}^{(2,f)}+f_{(zu)}^{(2,f)}\rp.
    \end{eqnarray}
After differentiating  (optimizing) with respect to $\gamma_{sq}$, $\p_2$, $\q_2$, and $\c_2$, we find $\c_2\rightarrow \infty$, $\gamma_{sq}\rightarrow 0$,
$\q_2\c_2^2\rightarrow \q_2^{(s)}$, and
\begin{eqnarray}\label{eq:negprac24a4}
f_{(z)}^{(2)}
 & \rightarrow &
 e^{\frac{\c_2^2-\q_2^{(s)}}{2}}
 \Bigg(\Bigg.
 e^{-\sqrt{\q_2^{(s)}}\h_1^{(3)}}
    + e^{\sqrt{\q_2^{(s)}}\h_1^{(3)}}
    \Bigg.\Bigg).
     \end{eqnarray}
We then observe that (\ref{eq:negprac24a2}) transforms into
\begin{eqnarray}\label{eq:negprac24a5}
\bar{h} & = &  -\frac{\sqrt{\p_2}\u_i^{(2,3)}+\kappa}{\sqrt{1-\p_2}}    \nonumber \\
\bar{B} & = & \frac{\c_2}{4\gamma_{sq}} \rightarrow \infty
\nonumber \\
\bar{C} & = & \sqrt{\p_2}\u_i^{(2,3)}+\kappa \nonumber \\
f_{(zd)}^{(2,f)}& = & \frac{e^{-\frac{\bar{B}\bar{C}^2}{2(1-\p_2)\bar{B} + 1}}}{2\sqrt{2(1-\p_2)\bar{B} + 1}}
\erfc\lp\frac{\bar{h}}{\sqrt{4(1-\p_2)\bar{B} + 2}}\rp
\rightarrow 0
\nonumber \\
f_{(zu)}^{(2,f)}& = & \frac{1}{2}\erfc\lp-\frac{\bar{h}}{\sqrt{2}}\rp
\rightarrow \frac{1}{2}\erfc\lp \frac{\sqrt{\p_2}\u_i^{(2,3)}+\kappa}{\sqrt{2}\sqrt{1-\p_2}} \rp.
   \end{eqnarray}
Combining (\ref{eq:negprac24}), (\ref{eq:negprac24a1}), (\ref{eq:negprac24a3}), (\ref{eq:negprac24a4}), and (\ref{eq:negprac24a5}), we obatin
\begin{eqnarray}\label{eq:negprac24a6}
-f_{sq}^{(2,f)}(\infty) & = &     \bar{\psi}_{rd}(\p,\q,\c,\gamma_{sq},1,1,-1) \nonumber \\
  & = &  \frac{1}{2}
(1-\p_2\q_2)\c_2
  - \frac{1}{\c_2}\mE_{{\mathcal U}_3}\log\lp \mE_{{\mathcal U}_2} e^{\c_2|\sqrt{1-\q_2}\h_1^{(2)} +\sqrt{\q_2}\h_1^{(3)} |}\rp \nonumber \\
& &   + \gamma_{sq}
 -\alpha\frac{1}{\c_2}\mE_{{\mathcal U}_3} \log\lp \mE_{{\mathcal U}_2} e^{-\c_2\frac{\max(\sqrt{1-\p_2}\u_1^{(2,2)}+\sqrt{\p_2}\u_1^{(2,3)},0)^2}{4\gamma_{sq}}}\rp \nonumber \\
 & = &  \frac{1}{2}
(1-\p_2\q_2)\c_2
  - \frac{1}{\c_2}\mE_{{\mathcal U}_3}\log\lp f_{(z)}^{(2)} \rp   + \gamma_{sq}
 -\alpha\frac{1}{\c_2}\mE_{{\mathcal U}_3} \log\lp f_{(zd)}^{(2,f)}+f_{(zu)}^{(2,f)}\rp \nonumber\\
 & \rightarrow & \frac{1}{2}
(1-\p_2\q_2)\c_2
  - \frac{1}{\c_2}\mE_{{\mathcal U}_3}\log\lp
 e^{\frac{\c_2^2-\q_2^{(s)}}{2}}
 \Bigg(\Bigg.
 e^{-\sqrt{\q_2^{(s)}}\h_1^{(3)}}
    + e^{\sqrt{\q_2^{(s)}}\h_1^{(3)}}
    \Bigg.\Bigg) \rp  \nonumber \\
& &
 -\alpha\frac{1}{\c_2}\mE_{{\mathcal U}_3} \log\lp   \frac{1}{2}\erfc\lp \frac{\sqrt{\p_2}\u_1^{(2,3)}+\kappa}{\sqrt{2}\sqrt{1-\p_2}} \rp  \rp \nonumber\\
 & \rightarrow & \frac{1}{2}
\frac{(1-\p_2)\q_2^{(s)}}{\c_2}
  - \frac{1}{\c_2}\mE_{{\mathcal U}_3}\log\lp 2 \cosh \lp \sqrt{\q_2^{(s)}}\h_1^{(3)}\rp \rp
  \nonumber \\
& &
 -\alpha\frac{1}{\c_2}\mE_{{\mathcal U}_3} \log\lp   \frac{1}{2}\erfc\lp \frac{\sqrt{\p_2}\u_1^{(2,3)}+\kappa}{\sqrt{2}\sqrt{1-\p_2}} \rp  \rp. \nonumber\\
    \end{eqnarray}
 Differentiating with respect to $\q_2^{(s)}$ and $\p_2$ then gives
\begin{eqnarray}\label{eq:negprac24a7}
  \frac{d  \bar{\psi}_{rd}(\p,\q,\c,\gamma_{sq},1,1,-1)}{d\q_2}
 &\rightarrow & \frac{1}{2}
\frac{(1-\p_2)}{\c_2}
  - \frac{1}{\c_2}\mE_{{\mathcal U}_3} \lp  \frac{\h_1^{(3)}}{2\sqrt{\q_2^{(s)}}} \tanh \lp \sqrt{\q_2^{(s)}}\h_1^{(3)} \rp
 \rp,  
    \end{eqnarray}
and
\begin{eqnarray}\label{eq:negprac24a8}
  \frac{d  \bar{\psi}_{rd}(\p,\q,\c,\gamma_{sq},1,1,-1)}{d\p_2}
 & \rightarrow & -
\frac{\q_2^{(s)}}{2\c_2}
 -\alpha\frac{1}{\c_2}\mE_{{\mathcal U}_3} \frac{-\frac{d }{d\p_2}\lp \erf\lp \frac{\sqrt{\p_2}\u_1^{(2,3)}+\kappa}{\sqrt{2}\sqrt{1-\p_2}} \rp  \rp}
 {\lp  \erfc\lp \frac{\sqrt{\p_2}\u_1^{(2,3)}+\kappa}{\sqrt{2}\sqrt{1-\p_2}} \rp  \rp} \nonumber \\
  & \rightarrow & -
\frac{\q_2^{(s)}}{2\c_2}
 -\alpha\frac{1}{\c_2}\mE_{{\mathcal U}_3} \lp  \frac{-\frac{2 }{\sqrt{\pi}}e^{-\lp \frac{\sqrt{\p_2}\u_1^{(2,3)}+\kappa}{\sqrt{2}\sqrt{1-\p_2}} \rp^2 } }
 {\lp  \erfc\lp \frac{\sqrt{\p_2}\u_1^{(2,3)}+\kappa}{\sqrt{2}\sqrt{1-\p_2}} \rp  \rp}
 \frac{d }{d\p_2} \lp \frac{\sqrt{\p_2}\u_1^{(2,3)}+\kappa}{\sqrt{2}\sqrt{1-\p_2}} \rp  \rp. \nonumber \\
    \end{eqnarray}
We then also find
\begin{eqnarray}\label{eq:negprac24a9}
  \frac{d }{d\p_2} \lp \frac{\sqrt{\p_2}\u_1^{(2,3)}+\kappa}{\sqrt{2}\sqrt{1-\p_2}} \rp
  =
   \frac{\frac{1}{2\sqrt{\p_2}}\u_1^{(2,3)}}{\sqrt{2}\sqrt{1-\p_2}}
+   \frac{\sqrt{\p_2}\u_1^{(2,3)}+\kappa}{2\sqrt{2}\sqrt{1-\p_2}^3}
=  \frac{\u_1^{(2,3)}+\sqrt{\p_2}\kappa}{2\sqrt{2}\sqrt{\p_2}\sqrt{1-\p_2}^3}.
  \end{eqnarray}
From (\ref{eq:negprac24a7}), we obtain
\begin{eqnarray}\label{eq:negprac24a10}
  \frac{d  \bar{\psi}_{rd}(\p,\q,\c,\gamma_{sq},1,1,-1)}{d\q_2}=0
 & \Longrightarrow &
 \p_2 \rightarrow 1 - \mE_{{\mathcal U}_3} \lp  \frac{\h_1^{(3)}}{\sqrt{\q_2^{(s)}}} \tanh \lp \sqrt{\q_2^{(s)}}\h_1^{(3)} \rp
 \rp  \triangleq \psi_{\p}(\q_2^{(s)}).  \nonumber\\
    \end{eqnarray}
Utilizing (\ref{eq:negprac24a6}), we have that $-f_{sq}^{(2,f)}(\infty)=\bar{\psi}_{rd}(\p,\q,\c,\gamma_{sq},1,1,-1) =0$ implies
\begin{eqnarray}\label{eq:negprac24a11}
   \bar{\psi}_{rd}(\p,\q,\c,\gamma_{sq},1,1,-1) =0
     \quad \Longrightarrow \quad
\alpha & \rightarrow & \frac{
\frac{(1-\p_2)\q_2^{(s)}}{2}
  - \mE_{{\mathcal U}_3}\log\lp  2\cosh \lp \sqrt{\q_2^{(s)}}\h_1^{(3)}\rp \rp
}
{   \mE_{{\mathcal U}_3} \log\lp   \frac{1}{2}\erfc\lp \frac{\sqrt{\p_2}\u_1^{(2,3)}+\kappa}{\sqrt{2}\sqrt{1-\p_2}} \rp  \rp} \nonumber \\
& \rightarrow &
\frac{
\frac{(1-\psi_{\p}(\q_2^{(s)}))\q_2^{(s)}}{2}
  - \mE_{{\mathcal U}_3}\log\lp  2\cosh \lp \sqrt{\q_2^{(s)}}\h_1^{(3)}\rp \rp
}
{   \mE_{{\mathcal U}_3} \log\lp   \frac{1}{2}\erfc\lp \frac{\sqrt{\psi_{\p}(\q_2^{(s)})}\u_1^{(2,3)}+\kappa}{\sqrt{2}\sqrt{1-\psi_{\p}(\q_2^{(s)})}} \rp  \rp} \triangleq \psi_{\alpha}(\q_2^{(s)}) . \nonumber\\
    \end{eqnarray}
From (\ref{eq:negprac24a8}) and (\ref{eq:negprac24a9}), we have
\begin{eqnarray}\label{eq:negprac24a12}
  \frac{d  \bar{\psi}_{rd}(\p,\q,\c,\gamma_{sq},1,1,-1)}{d\p_2} =0
 & \Longrightarrow &
\q_2^{(s)}
\rightarrow \alpha \mE_{{\mathcal U}_3} \lp  \frac{\sqrt{\frac{2}{\pi}}e^{-\lp \frac{\sqrt{\p_2}\u_1^{(2,3)}+\kappa}{\sqrt{2}\sqrt{1-\p_2}} \rp^2 } }
 {\lp  \erfc\lp \frac{\sqrt{\p_2}\u_1^{(2,3)}+\kappa}{\sqrt{2}\sqrt{1-\p_2}} \rp  \rp}
\lp \frac{\u_1^{(2,3)}+\sqrt{\p_2}\kappa}{\sqrt{\p_2}\sqrt{1-\p_2}^3} \rp \rp.
    \end{eqnarray}
A combination of (\ref{eq:negprac24a10}), (\ref{eq:negprac24a11}), and (\ref{eq:negprac24a12}) then gives
\begin{eqnarray}\label{eq:negprac24a13}
 \q_2^{(s)}
\rightarrow \psi_{\alpha}(\q_2^{(s)})
 \mE_{{\mathcal U}_3} \lp \frac{\sqrt{\frac{2}{\pi}}e^{-\lp \frac{\sqrt{\psi_{\p}(\q_2^{(s)})}\u_1^{(2,3)}+\kappa}{\sqrt{2}\sqrt{1-\psi_{\p}(\q_2^{(s)})}} \rp^2 } }
 {\lp  \erfc\lp \frac{\sqrt{\psi_{\p}(\q_2^{(s)})}\u_1^{(2,3)}+\kappa}{\sqrt{2}\sqrt{1-\psi_{\p}(\q_2^{(s)})}} \rp  \rp}
\lp \frac{\u_1^{(2,3)}+\sqrt{\psi_{\p}(\q_2^{(s)})}\kappa}{\sqrt{\psi_{\p}(\q_2^{(s)})}\sqrt{1-\psi_{\p}(\q_2^{(s)})}^3} \rp \rp
\triangleq \psi_{\q}(\q_2^{(s)}).
    \end{eqnarray}
Assume that  $\hat{\q}_2^{(s)}$ satisfies
\begin{eqnarray}\label{eq:negprac24a14}
 \hat{\q}_2^{(s)}-\psi_{\q}(\hat{\q}_2^{(s)})=0,
    \end{eqnarray}
    and let
\begin{eqnarray}\label{eq:negprac24a14a0}
  \psi_{\p}(\hat{\q}_2^{(s)}) = 1 - \mE_{{\mathcal U}_3} \lp  \frac{\h_1^{(3)}}{\sqrt{\hat{\q}_2^{(s)}}} \tanh \lp \sqrt{\hat{\q}_2^{(s)}}\h_1^{(3)} \rp
 \rp= \mE_{{\mathcal U}_3} \lp  \tanh \lp \sqrt{\hat{\q}_2^{(s)}}\h_1^{(3)} \rp^2
 \rp.
    \end{eqnarray}
 Provided that $\psi_{\p}(\hat{\q}_2^{(s)})\in [0,1]$ (which trivially holds as $\tanh^2(\cdot)\in[0,1]$), then
\begin{eqnarray}\label{eq:negprac24a15}
\alpha_c^{(2,f)}(\kappa)
\rightarrow
\frac{
\frac{(1-\psi_{\p}(\hat{\q}_2^{(s)}))\hat{\q}_2^{(s)}}{2}
  - \mE_{{\mathcal U}_3}\log\lp  2\cosh \lp \sqrt{\hat{\q}_2^{(s)}}\h_1^{(3)}\rp \rp
}
{   \mE_{{\mathcal U}_3} \log\lp   \frac{1}{2}\erfc\lp \frac{\sqrt{\psi_{\p}(\hat{\q}_2^{(s)})}\u_1^{(2,3)}+\kappa}{\sqrt{2}\sqrt{1-\psi_{\p}(\hat{\q}_2^{(s)})}} \rp  \rp}.
\end{eqnarray}
Specializing to $\kappa=0$, we find $\psi_{\p}(\hat{\q}_2^{(s)}) = 0.5639...$ (which clearly belongs to $[0,1]$ interval) and
\begin{equation}\label{eq:negprac25}
\hspace{-2in}(\mbox{\emph{full} second level:}) \qquad \qquad  a_c^{(2,f)}(0) \rightarrow
 \bl{\mathbf{0.8330786...}}.
  \end{equation}
For the special case, $\kappa=0$, all the relevant quantities for the first full (1-sfl RDT), second partial (2-spf RDT), and second full (2-sfl RDT) level of lifting are systematically shown in Table \ref{tab:tab1}.
\begin{table}[h]
\caption{$r$-sfl RDT parameters; binary $\pm 1$ perceptron capacity;  $\hat{\c}_1\rightarrow 1$; $\kappa=0$; $n,\beta\rightarrow\infty$}\vspace{.1in}
\centering
\def\arraystretch{1.2}
\begin{tabular}{||l||c||c|c||c|c||c||c||}\hline\hline
 \hspace{-0in}$r$-sfl RDT                                             & $\hat{\gamma}_{sq}$    &  $\hat{\p}_2$ & $\hat{\p}_1`$     & $\hat{\q}_2^{(s)}\rightarrow \hat{\q}_2\hat{\c}_2^2$  & $\hat{\q}_1$ &  $\hat{\c}_2$    & $\alpha_c^{(r)}(0)$  \\ \hline\hline
$1$-sfl RDT                                      & $0.3989$ &  $0$  & $\rightarrow 1$   & $0$ & $\rightarrow 1$
 &  $\rightarrow 0$  & \bl{$\mathbf{1.2732}$} \\ \hline\hline
 $2$-spl RDT                                      & $0$ &  $0$ & $\rightarrow 1$ &  $0$ & $\rightarrow 1$ &   $\rightarrow \infty$   & \bl{$\mathbf{1}$} \\ \hline
  $2$-sfl RDT                                      & $0$  & $0.5639$ & $\rightarrow 1$ &  $2.5764$ & $\rightarrow 1$
 &  $\rightarrow \infty$   & \bl{$\mathbf{0.8331}$}  \\ \hline\hline
  \end{tabular}
\label{tab:tab1}
\end{table}
Also, for the completeness, we show in Figure \ref{fig:fig0}, behavior of function $\q_2^{(s)}-\psi_{\q}(\q_2^{(s)})$, which indicates the uniqueness of $\hat{\q}_2^{(s)}$.
\begin{figure}[h]
\centering
\centerline{\includegraphics[width=0.6\linewidth]{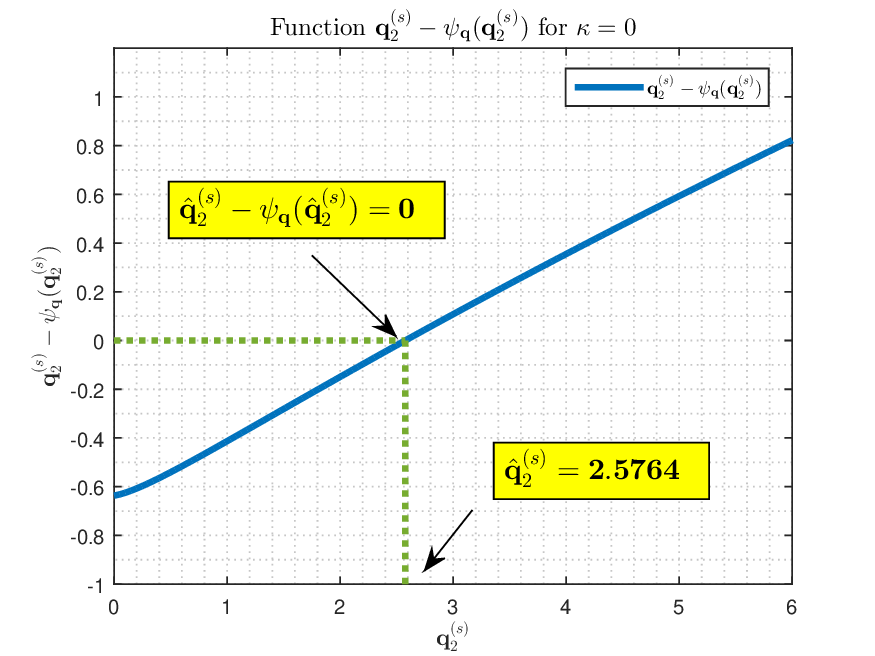}}
\caption{Uniqueness of $\hat{\q}_2^{(s)}$;  $\q_2^{(s)}-\psi_{\q}(\q_2^{(s)})$ as a function of $\q_2^{(s)}$; $\kappa=0$}
\label{fig:fig0}
\end{figure}

As we have found no further changes for $r=3$, the following theorem summarizes the above considerations.
\begin{theorem}
\label{thm:thmseclev}
  Assume the complete sfl RDT frame of\cite{Stojnicsflgscompyx23} with $r=2$. Let
  ${\mathcal F}(G,\b,\cX,\alpha)$ and $\xi_{feas}(0,\cX)$ be as in (\ref{eq:ex1}) and (\ref{eq:ex4}), respectively. Assume that the elements of $G\in\mR^{m\times n}$, $\h_1^{(3)}$, and $\u_1^{(2,3)}$ are i.i.d. standard normals and let
 \begin{eqnarray}
\alpha & = &    \lim_{n\rightarrow \infty} \frac{m}{n}  \nonumber \\
\alpha_c^{(b)}(\kappa) & \triangleq & \max \left \{\alpha |\hspace{.08in}  \lim_{n\rightarrow\infty}\mP_G \lp \xi_{perc}\lp0,\left \{-\frac{1}{\sqrt{n}},\frac{1}{\sqrt{n}}\right \}^n\rp\triangleq \xi_{feas}\lp0,\left \{-\frac{1}{\sqrt{n}},\frac{1}{\sqrt{n}}\right \}^n\rp>0\rp\longrightarrow 1 \right\} \nonumber \\
& = & \max \left \{\alpha |\hspace{.08in}  \lim_{n\rightarrow\infty}\mP_G\lp {\mathcal F}\lp G,\kappa\1,\left \{-\frac{1}{\sqrt{n}},\frac{1}{\sqrt{n}}\right \}^n,\alpha\rp \hspace{.07in}\mbox{is feasible} \rp\longrightarrow 1 \right \}.
  \label{eq:thmseclevex1}
\end{eqnarray}
Set
\begin{eqnarray}\label{eq:thmseclevex2}
  \psi_{\p}(\q_2^{(s)}) &\triangleq & 1 - \mE_{{\mathcal U}_3} \lp  \frac{\h_1^{(3)}}{\sqrt{\q_2^{(s)}}} \tanh \lp \sqrt{\q_2^{(s)}}\h_1^{(3)} \rp
 \rp  \nonumber\\
\psi_{\alpha}(\q_2^{(s)}) & \triangleq &
 \frac{
\frac{(1-\psi_{\p}(\q_2^{(s)}))\q_2^{(s)}}{2}
  - \mE_{{\mathcal U}_3}\log\lp  2\cosh \lp \sqrt{\q_2^{(s)}}\h_1^{(3)}\rp \rp
}
{   \mE_{{\mathcal U}_3} \log\lp   \frac{1}{2}\erfc\lp \frac{\sqrt{\psi_{\p}(\q_2^{(s)})}\u_1^{(2,3)}+\kappa}{\sqrt{2}\sqrt{1-\psi_{\p}(\q_2^{(s)})}} \rp  \rp} \nonumber\\
\psi_{\q}(\q_2^{(s)}) & \triangleq & \psi_{\alpha}(\q_2^{(s)})
 \mE_{{\mathcal U}_3} \lp \frac{\sqrt{\frac{2}{\pi}}e^{-\lp \frac{\sqrt{\psi_{\p}(\q_2^{(s)})}\u_1^{(2,3)}+\kappa}{\sqrt{2}\sqrt{1-\psi_{\p}(\q_2^{(s)})}} \rp^2 } }
 {\lp  \erfc\lp \frac{\sqrt{\psi_{\p}(\q_2^{(s)})}\u_1^{(2,3)}+\kappa}{\sqrt{2}\sqrt{1-\psi_{\p}(\q_2^{(s)})}} \rp  \rp}
\lp \frac{\u_1^{(2,3)}+\sqrt{\psi_{\p}(\q_2^{(s)})}\kappa}{\sqrt{\psi_{\p}(\q_2^{(s)})}\sqrt{1-\psi_{\p}(\q_2^{(s)})}^3} \rp \rp
    \end{eqnarray}
Assuming that $\hat{\q}_2^{(s)}$ satisfies
\begin{eqnarray}\label{eq:thmseclevex3}
 \hat{\q}_2^{(s)}-\psi_{\q}(\hat{\q}_2^{(s)})=0,
    \end{eqnarray}
    then
 \begin{eqnarray}\label{eq:negprac24a15}
\alpha_c^{(b)}(\kappa)\rightarrow \alpha_c^{(2,f)}(\kappa)
\rightarrow
\psi_{\alpha}(\hat{\q}_2^{(s)}).
\end{eqnarray}
\end{theorem}

In Table \ref{tab:tab2}, we show the full second level of lifting results over a range of $\kappa$ obtained based on the above theorem.
\begin{table}[h]
\caption{$2$-sfl RDT parameters; \emph{binary}, $\pm 1$, perceptron capacity $\alpha=\alpha_c^{(2,f)}(\kappa)$;  $\hat{\q}^{(s)}_2=\lim_{\hat{\c}_2\rightarrow\infty} \hat{\q}_2\hat{\c}_2^2$ }\vspace{.1in}
\centering
\def\arraystretch{1.2}
\begin{tabular}{||l||c|c|c|c|c|c|c|c|c|c||}\hline\hline
 \hspace{-0in}$\kappa$                                             & $\mathbf{-0.6}$    & $\mathbf{-0.4}$ & $\mathbf{-0.2}$ & $\mathbf{0}$        & $\mathbf{0.2}$  & $\mathbf{0.4}$   & $\mathbf{0.6}$  &  $\mathbf{0.8}$    & $\mathbf{1.0}$    & $\mathbf{1.2}$  \\ \hline\hline
$\hat{\p}_2$                                       & $0.4536$ & $0.4932$ & $0.5299$ & $0.5639$ & $0.5955$ & $ 0.6246 $ & $0.6517$ & $0.6767$ & $0.6999$ & $0.7213$
 \\ \hline
$\hat{\q}_{2}^{(s)}$\hspace{-.08in}                                      & $1.3837   $  & $ 1.7215   $  & $ 2.1162   $  & $ 2.5764   $  & $ 3.1120  $  & $  3.7349  $  & $  4.4586 $  & $   5.2984   $  & $ 6.2717 $  & $   7.3981   $  \\ \hline \hline
 $\alpha$                                      & $\bl{\mathbf{1.9570}}  $ & $ \bl{\mathbf{ 1.4468 }} $ & $  \bl{\mathbf{1.0888}} $ & $  \bl{\mathbf{ 0.8331 }} $ & $ \bl{\mathbf{ 0.6474 }}$ & $  \bl{\mathbf{ 0.5105}}  $ & $ \bl{\mathbf{ 0.4081}} $ & $  \bl{\mathbf{ 0.3304 }} $ & $ \bl{\mathbf{ 0.2707 }}$ & $   \bl{\mathbf{0.2243 }}$ \\ \hline\hline
 \end{tabular}
\label{tab:tab2}
\end{table}
On the other hand, for the same range of $\kappa$, we, in Table \ref{tab:tab3}, show the progression of the capacity as the level of lifting increases. Finally, the results from Table \ref{tab:tab3} are visualized in Figure \ref{fig:fig1}. For the completeness, we, in Figure \ref{fig:fig1}, also provide (as yellow dots) the capacity values obtained practically by running a simple gradient based algorithm.
\begin{table}[h]
\caption{\emph{Binary}, $\pm 1$, perceptron capacity $\alpha_c(\kappa)$ --- progression of $r$-sfl RDT mechanism}\vspace{.1in}
\centering
\def\arraystretch{1.2}
\begin{tabular}{||l||c|c|c|c|c|c|c|c|c||}\hline\hline
 \hspace{-0in}$\kappa$                                               & $\mathbf{-0.4}$ & $\mathbf{-0.2}$ & $\mathbf{0}$        & $\mathbf{0.2}$  & $\mathbf{0.4}$   & $\mathbf{0.6}$  &  $\mathbf{0.8}$    & $\mathbf{1.0}$    & $\mathbf{1.2}$  \\ \hline\hline
$\alpha_c^{(1,f)}(\kappa)$                    & $    2.5222 $  & $    1.7715   $  & $  1.2732 $  & $    0.9353 $  & $    0.7014 $  & $    0.5363    $  & $  0.4176  $  & $   0.3308  $  & $   0.2661
$
 \\ \hline
$\alpha_c^{(2,p)}(\kappa)$                                      & $  1.6407  $ & $   1.2695   $ & $  1.0000  $ & $   0.8006   $ & $  0.6506   $ & $  0.5358   $ & $  0.4176  $ & $   0.3308   $ & $  0.2661 $  \\ \hline \hline
 $\alpha_c^{(2,f)}(\kappa)$                                      & $ \bl{\mathbf{ 1.4468 }} $ & $  \bl{\mathbf{1.0888}} $ & $  \bl{\mathbf{ 0.8331 }} $ & $ \bl{\mathbf{ 0.6474 }}$ & $  \bl{\mathbf{ 0.5105}}  $ & $ \bl{\mathbf{ 0.4081}} $ & $  \bl{\mathbf{ 0.3304 }} $ & $ \bl{\mathbf{ 0.2707 }}$ & $   \bl{\mathbf{0.2243 }}$ \\ \hline\hline
 \end{tabular}
\label{tab:tab3}
\end{table}

\begin{figure}[h]
\centering
\centerline{\includegraphics[width=0.6\linewidth]{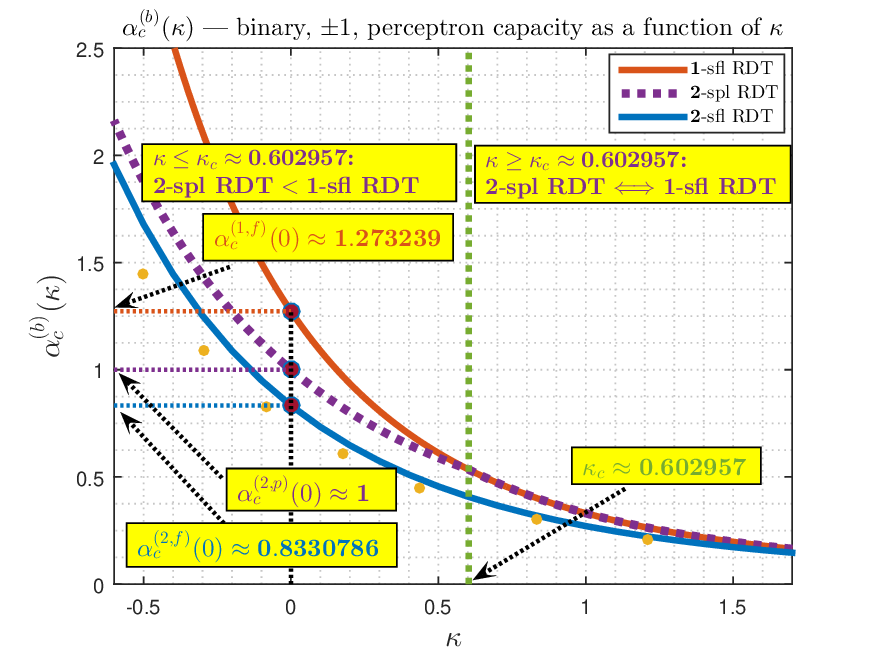}}
\caption{Binary perceptron capacity as a function of $\kappa$}
\label{fig:fig1}
\end{figure}

It is also interesting to note that the capacity results of the above theorem (as well as the results presented in Table \ref{tab:tab3} and Figure \ref{fig:fig1})
are fully matching both those obtained through replica theory based considerations in \cite{KraMez89} and those obtained through rigorous second moment lower bounding analysis from \cite{DingSun19,NakSun23}. Particularly fascinating is that, despite the fact that the so-called partition functions, as the main objects of study in all these papers, are completely different from ours, the final conclusions and underlying key functions needed to reach them are identical.

\subsubsection{Modulo-$\m$ sfl RDT}
\label{sec:posmodm}

 We also note that the above can be repeated relying on the so-called modulo-$m$ sfl RDT frame of \cite{Stojnicsflgscompyx23}. Instead of Theorem \ref{thme:negthmprac1}, one then basically has the following theorem.
 \begin{theorem}
  \label{thme:negthmprac2}
  Assume the setup of Theorem \ref{thme:negthmprac1} and instead of the complete, assume the modulo-$\m$ sfl RDT setup of \cite{Stojnicsflgscompyx23}.
   Let the ``fixed'' parts of $\hat{\p}$, $\hat{\q}$, and $\hat{\c}$ satisfy $\hat{\p}_1\rightarrow 1$, $\hat{\q}_1\rightarrow 1$, $\hat{\c}_1\rightarrow 1$, $\hat{\p}_{r+1}=\hat{\q}_{r+1}=\hat{\c}_{r+1}=0$, and let the ``non-fixed'' parts of $\hat{\p}_k$, and $\hat{\q}_k$ ($k\in\{2,3,\dots,r\}$) be the solutions of the following system of equations
  \begin{eqnarray}\label{eq:negthmprac2eq1}
   \frac{d \bar{\psi}_{rd}(\p,\q,\c,\gamma_{sq},1,1,-1)}{d\p} =  0 \nonumber \\
   \frac{d \bar{\psi}_{rd}(\p,\q,\c,\gamma_{sq},1,1,-1)}{d\q} =  0 \nonumber \\
    \frac{d \bar{\psi}_{rd}(\p,\q,\c,\gamma_{sq},1,1,-1)}{d\gamma_{sq}} =  0.
 \end{eqnarray}
 Consequently, let
\begin{eqnarray}\label{eq:negthmprac2eq2}
c_k(\hat{\p},\hat{\q})  & = & \sqrt{\hat{\q}_{k-1}-\hat{\q}_k} \nonumber \\
b_k(\hat{\p},\hat{\q})  & = & \sqrt{\hat{\p}_{k-1}-\hat{\p}_k}.
 \end{eqnarray}
 Then
 \begin{align}
-f_{sq,\m}(\infty)
& \leq     \max_{\c} \frac{1}{2}    \sum_{k=2}^{r+1}\Bigg(\Bigg.
   \hat{\p}_{k-1}\hat{\q}_{k-1}
   -\hat{\p}_{k}\hat{\q}_{k}
  \Bigg.\Bigg)
\c_k
  - \varphi(D_1^{(bin)}(c_k(\hat{\p},\hat{\q})),\c) + \hat{\gamma}_{sq} - \alpha\varphi(-D_1^{(sph)}(b_k(\hat{\p},\hat{\q})),\c). \nonumber \\
  \label{eq:negthmprac2eq3}
\end{align}
\end{theorem}
\begin{proof}
Follows from the previous discussion, Theorems \ref{thm:thmsflrdt1} and \ref{thme:negthmprac1}, Corollary \ref{cor:cor1}, and the sfl RDT machinery presented in \cite{Stojnicnflgscompyx23,Stojnicsflgscompyx23,Stojnicflrdt23}.
\end{proof}
We conducted the numerical evaluations using the modulo-$\m$ results of the above theorem and have not found any scenario where the inequality in (\ref{eq:negthmprac2eq3}) is not tight. In other words, we have found that $f_{sq}^{(r)}(\infty)=f_{sq,\m}^{(r)}(\infty)$. This indicates that the \emph{stationarity} over $\c$ is actually of the \emph{maximization} type.

\section{Conclusion}
\label{sec:conc}

The statistical capacity of the classical binary perceptrons with general thresholds $\kappa$ is the main subject of study in this paper. First, we recognize the connection between the statistical perceptron problems as general \emph{random feasibility problems} (rfps) and the \emph{random duality theory} (RDT) \cite{StojnicCSetam09,StojnicICASSP10var,StojnicRegRndDlt10,StojnicGardGen13,StojnicICASSP09}. As the RDT is strongly related to the bilinearly indexed (bli) random processes, we then proceed by further recognizing that studying rfps is directly related to studying bli processes. Relying on a recent progress made  in \cite{Stojnicsflgscompyx23,Stojnicnflgscompyx23} in studying these processes, \cite{Stojnicflrdt23} established a \emph{fully lifted} random duality theory (fl RDT) counterpart to the RDT from \cite{StojnicCSetam09,StojnicICASSP10var,StojnicRegRndDlt10,StojnicGardGen13,StojnicICASSP09}. Utilizing the fl RDT and its a particular \emph{stationarized} variant (called sfl RDT), we create a general framework for studying the perceptrons' capacities. To ensure that the framework and ultimately the entire fl RDT machinery are practically fully operational, a sequence of  underlying numerical evaluations is required as well. After successfully conducting these evaluations,  we uncover that the capacity characterizations are achieved on the second (first non-trivial) level of stationarized full lifting. The obtained results completely match both the replica symmetry breaking predictions obtained through statistical physics methods in \cite{KraMez89} and the rigorous lower bounds obtained through a second moment based analysis in \cite{DingSun19,NakSun23}. In particular, for the most famous, zero-threshold ($\kappa=0$) scenario, we uncover the well known $\approx 0.8330786...$ scaled capacity.

The presented methodology is very generic and easily allows for various extensions and generalizations. These include many related to \emph{random feasibility problems} (rfps) such as  \emph{spherical} perceptrons   (see, e.g., \cite{StojnicGardGen13,StojnicGardSphErr13,StojnicGardSphNeg13,GarDer88,Gar88,Schlafli,Cover65,Winder,Winder61,Wendel,Cameron60,Joseph60,BalVen87,Ven86,SchTir02,SchTir03}), \emph{symmetric binary} perceptrons (see, e.g., \cite{StojnicGardGen13,GarDer88,Gar88,StojnicDiscPercp13,KraMez89,GutSte90,KimRoc98}), the \emph{sectional} compressed sensing $\ell_1$ phase transitions (see, e.g., \cite{StojnicCSetam09,StojnicLiftStrSec13}) as well as many related to generic random structures discussed in \cite{Stojnicnflgscompyx23,Stojnicsflgscompyx23,Stojnicflrdt23}. Similarly to the consideration discussed here, such extensions often require a bit of technical,  problem specific, adjustments that we discuss in separate papers.

As mentioned in \cite{Stojnicflrdt23,Stojnichopflrdt23}, the sfl RDT considerations do not require the standard Gaussianity  assumption of the random primals. Utilizing, for example, the Lindeberg variant \cite{Lindeberg22} of the central limit theorem, the obtained sfl RDT results can quickly be extended to a wide range of different statistics. In that regard, we mention \cite{Chatterjee06} as a particularly elegant utilization of the Lindenberg approach.

\begin{singlespace}
\bibliographystyle{plain}
\bibliography{nflgscompyxRefs}
\end{singlespace}

\end{document}